\documentclass[11pt]{amsart}

\usepackage{amsmath,amscd}
\usepackage{amssymb}
\usepackage{amsthm}
\usepackage{enumitem}

\usepackage{graphicx,tikz}
\usetikzlibrary{patterns}
\usetikzlibrary{decorations.pathmorphing}

\usepackage{mathtools}

\usepackage{mathrsfs}
\usepackage[ocgcolorlinks, linkcolor=blue]{hyperref}

\usepackage{calc}
             {\begin{list}{\arabic{enumi}.}{\usecounter{enumi}%
              \setlength{\labelsep}{0.5em}%
              \settowidth{\labelwidth}{(\arabic{enumi})}%
              \setlength{\leftmargin}{\labelwidth+\labelsep}}}%
             {\end{list}}

\usepackage{comment}

\DeclareRobustCommand{\SkipTocEntry}[5]{}

\usepackage{xcolor}
\definecolor{LOcolor}{RGB}{150,100,0}

\DeclareMathOperator{\WF}{WF}

\newtheorem{Theorem}{Theorem}[section]

\newtheorem{Lemma}[Theorem]{Lemma}

\newtheorem{Proposition}[Theorem]{Proposition}

\theoremstyle{definition}

\newtheorem{Remark}[Theorem]{Remark}

\numberwithin{equation}{section}

\newcommand{\norm}[1]{\lVert #1 \rVert}         

\newcommand{\ol}[1]{\overline{#1}}

\newcommand{\p}{\partial}

\newcommand{\tre}{\textcolor{red}}

\newcounter{sidenote}

\allowdisplaybreaks

\begin{document}

\title{Propagation of singularities and inverse problems for the viscoacoustic wave equation} 

\author[G. Covi]{Giovanni Covi}
\address{Department of Mathematics and Statistics, University of Jyv\"askyl\"a, PO Box 35, 40014 Jyv\"askyl\"a, Finland}
\email{giovanni.g.covi@jyu.fi} 

\author[M.V. de Hoop]{Maarten V. de Hoop}
\address{Department of Computational and Applied Mathematics, Rice University, Houston, TX, USA}
\email{mvd2@rice.edu}

\author[M. Salo]{Mikko Salo}
\address{Department of Mathematics and Statistics, University of Jyv\"askyl\"a, PO Box 35, 40014 Jyv\"askyl\"a, Finland}
\email{mikko.j.salo@jyu.fi}



\begin{abstract} 
We study an inverse problem for the viscoacoustic wave equation, an integro-differential model describing wave propagation in viscoacoustic media with memory in the leading order term. The medium is characterized by a spatially varying sound speed and a space-time dependent memory kernel. Assuming that waves are generated by sources supported outside the region of interest, we consider exterior measurements encoded by the source-to-solution map. To study this inverse problem, we construct solutions concentrating near fixed geodesics and establish a corresponding propagation of singularities result for the semiclassical wave front set. These results are valid without any restriction on the underlying sound speed. Then, under certain geometric conditions, we prove that the exterior data uniquely determine not just the sound speed inside the domain but also all time derivatives at zero of the memory kernel. This involves a reduction to the lens rigidity and geodesic ray transform inverse problems. 
As an application, we establish uniqueness for the recovery of variable parameters in the extended Maxwell model.

\medskip\medskip

\noindent
Nous étudions un problème inverse pour l’équation des ondes viscoacoustiques, un modèle intégro-différentiel décrivant la propagation des ondes dans des milieux viscoacoustiques avec mémoire apparaissant dans le terme d’ordre principal. Le milieu est caractérisé par une vitesse du son dépendant de la variable spatiale et par un noyau de mémoire dépendant de l’espace et du temps. En supposant que les ondes sont engendrées par des sources supportées à l’extérieur de la région d’intérêt, nous considérons des mesures extérieures encodées par l’opérateur source-solution. Pour l’étude de ce problème inverse, nous construisons des solutions se concentrant au voisinage de géodésiques fixées et établissons des résultat de propagation des singularités pour le front d’onde semi-classique associé. Ces résultats sont obtenus sans hypothèse restrictive sur la vitesse du son sous-jacente. Sous certaines conditions géométriques, nous démontrons ensuite que les données extérieures déterminent de manière unique non seulement la vitesse du son à l’intérieur du domaine, mais également toutes les dérivées temporelles en zéro du noyau de mémoire. La démonstration repose sur une réduction aux problèmes inverses de rigidité des lentilles et de la transformée des rayons géodésiques. En application, nous établissons un résultat d’unicité pour la reconstruction de paramètres variables dans le modèle de Maxwell étendu.
\end{abstract}

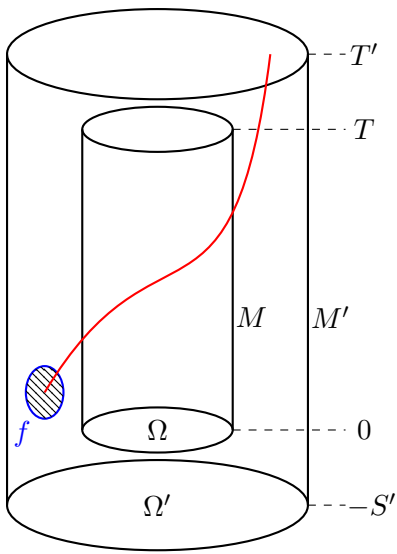
\maketitle

\tableofcontents

\section{Introduction}

In this article we study an inverse problem in viscoacoustics. This can be viewed as a scalar model for viscoelastic media. Our methods would likely work also for the viscoelastic equation, but for simplicity we focus our attention to the scalar case in this article.

Let $\Omega\subset\mathbb R^n$, $n\geq 2$, be an open bounded set with smooth boundary, and let $T> 0$. The set $\Omega$ represents a viscoacoustic medium within which acoustic waves originating from $\Omega_e:=\mathbb R^n\setminus\Omega$ propagate during the time interval $[0,T]$. The medium is characterized by two physical quantities: the wave speed $c$, which is independent of time, and the memory kernel $G$, which is a function of both space and time. We make the following geometric assumptions:
\begin{enumerate}[label=(A\arabic*)]
    \item The manifold $(\ol{\Omega}, c(x)^{-2}g_0)$, where $g_0$ represents the Euclidean metric, is strictly convex;
    \item If $n=2$, we require that  $(\ol{\Omega}, c(x)^{-2}g_0)$ is a simple manifold;
    \item If $n\geq 3$, we require that $(\ol{\Omega}, c(x)^{-2}g_0)$ is foliated by strictly convex hypersurfaces, in the sense of \cite{SUV15}.
\end{enumerate}
 The pressure $u$ evolves according to the viscoacoustic equation
\begin{equation}\label{viscoacoustic-equation}
    Pu := -\frac 12\left(\p_t^2 u - \nabla \cdot (c(x) \nabla u) + \int_0^t \nabla \cdot (G(x,t-s) \nabla u(x,s)) \,ds\right)=0,
\end{equation}
whose derivation is recalled in Appendix \ref{sec_appendix}. The operator $P$ consists of a purely elastic part $Q$ (the wave operator), and an integro-differential operator $R$ (the memory operator):
$$Q := -\frac 12(\p_t^2 - \nabla \cdot c(x) \nabla), \qquad Ru := -\frac 12\int_0^t \nabla \cdot (G(x,t-s) \nabla u(x,s)) \,ds.$$
The second part in particular originates from the assumption that the stress within the medium at the spacetime point $(x,t)$ depends on the whole history of the quantity $\nabla u$ at $x$, that is on $\nabla u(x,s)$ for all $s\in[0,t]$. Thus, $R$ accounts for the viscous behaviour of the medium. 
\begin{center}
\begin{figure}[t]
 \begin{tikzpicture}

  \draw[thick] (0,1) ellipse (2 and 0.6); 
  \draw[thick] (-2,1) -- (-2,-5);          
  \draw[thick] (2,1) -- (2,-5);            
  \draw[thick] (0,-5) ellipse (2 and 0.6); 

  \draw[thick] (0,0) ellipse (1 and 0.3); 
  \draw[thick] (-1,0) -- (-1,-4);                      
  \draw[thick] (1,0) -- (1,-4);                        
  \draw[thick] (0,-4) ellipse (1 and 0.3);             

  \draw[dashed] (2,-5) -- (2.5,-5);
  \draw[dashed] (2,1) -- (2.5,1);
  \draw[dashed] (1,0) -- (2.5,0);
  \draw[dashed] (1,-4) -- (2.5,-4);

  \draw[blue,thick] (-1.5,-3.5) ellipse (0.25 and 0.35); 
  \fill[color=blue, pattern=north west lines] (-1.5,-3.5) ellipse (0.25 and 0.35);
  \draw[red,thick]
    (-1.5,-3.5) .. controls (0,-1) and (1,-3) .. (1.5,1);

  \node at (0,-5) {$\Omega'$};
  \node at (0,-4) {$\Omega$};
  \node at (2.3,-2.5) {$M'$};
  \node at (1.25,-2.5) {$M$};
  \node at (-1.8,-4.05) {{\color{blue}$f$}};
  \node at (2.85,-5) {$-S'$};
  \node at (2.75,1) {$T'$};
  \node at (2.75,0) {$T$};
  \node at (2.75,-4) {$0$};

\end{tikzpicture}
\end{figure}
\end{center}
\vspace{-5mm}

We want to consider waves originating from sources $f$ supported in the exterior. To this end, let $M:= [0,T]\times\Omega$ and $M':=[-S', T']\times\Omega'$ be spacetime cylinders such that $M\Subset M'$, and consider $f\in C^\infty_c(M'\setminus M)$. We prove in Proposition \ref{prop:well-posedness} that, under convenient assumptions for the functions $c$ and $G$, the problem
\begin{equation*}
     \begin{cases}
         Pu = f, & \mbox{ in } M', \\
         u(-S') = \dot u(-S') = 0, & \mbox{ in } \Omega'
     \end{cases}
 \end{equation*}
 is well-posed in $L^2([-S',T'], H^1_0(\Omega')) \cap H^1([-S',T'], L^2(\Omega'))$ for all sources $f\in L^2([-S',T'], H^{-1}(\Omega'))$. This allows us to define the source-to-solution map 
 $$\mathcal S: f \mapsto u|_{M'\setminus M},$$
 associating exterior sources to the waves they generate, as observed in the exterior of $\Omega$. In this framework, we ask the following inverse problem: 

 \emph{Are the wave speed $c$ and the memory kernel $G$ within $\Omega$ uniquely determined by exterior data in the form of the source-to-solution map $\mathcal S$?}

 Several results for this inverse problem were proved by Romanov and collaborators, see \cite{Rom-article, Romanov} and references therein. The result in \cite{Rom-article} states roughly that if the manifold $(\ol{\Omega}, c(x)^{-2} g_0)$ is simple, then $c$ and all derivatives $\p_t^j G(x,0)$ for $j \geq 0$ are determined by the source-to-solution map $\mathcal S$ when the sources are delta functions at points $y \in \p \Omega$ at time $t=0$. The proof is based on writing a singularity expansion for solutions generated by a delta source. From the boundary measurements of solutions one can read off the travel times of waves between boundary points (i.e.\ the boundary distance function) and certain integrals over geodesics involving the quantities $\p_t^j G(x,0)$. By invoking the solution of the boundary rigidity problem and invertibility of geodesic X-ray transforms on simple manifolds, one can determine $c$ and $\p_t^j G(x,0)$ in $\Omega$.

 We will consider the above inverse problem in cases where the wave speed may not be simple. Then there may be caustics, which are generically present in the seismic applications that we have in mind. In the presence of caustics, the conormal structure of the solution corresponding to a delta source exploited in \cite{Rom-article} breaks down, and in general it is not possible to read off the required information from such boundary data. We will take a more microlocal approach and employ oscillatory sources that generate Gaussian beam type solutions concentrating along a fixed null bicharacteristic. This requires a propagation result for wave-type singularities of the viscoacoustic equation. We establish such a result for the semiclassical wave front set, which is the appropriate setting for oscillatory solutions. It is interesting that the memory term, even though it contains second order derivatives, turns out not to affect the propagation of wave-type singularities. The effect of the memory term is visible only in the amplitudes, obtained via solving both transport and Volterra equations, of the Gaussian beam solutions.
 
 Solutions of the viscoacoustic equation can also have stationary singularities which do not move (see Remark \ref{stationary-sing} for a discussion). However, we only use wave-type singularities in the solution of the inverse problem, since these are the singularities that travel through the unknown part of medium and collect information on the coefficients there.

\subsection{Main results}
 We begin with the construction of  special approximate solutions to the viscoacoustic equation, called \emph{Gaussian beam quasimodes} given as superpositions of wave packets as in \cite{OSSU}. We have obtained the following result:

\begin{Theorem}\label{GO-construction-general}
     Let $\Omega \Subset\Omega'\subset\mathbb R^n$, $n\geq 2$, be open bounded sets, and assume $T>0$. Define the spacetime cylinders $M:= [0,T]\times\overline\Omega$ and $Z:=[0, T]\times\overline\Omega'$. Let $c \in C^\infty(\ol{\Omega}')$ be a positive function and consider a bicharacteristic $\gamma(\sigma) = (z(\sigma),\zeta(\sigma))$, $\sigma\in[0,T]$, of the wave operator $Q$, where $z(\sigma) = (x(\sigma), \sigma)$, $x(\sigma)$ is a geodesic such that $x([0,T]) \subset \ol{\Omega}'$ and $x(0), x(T)\in \Omega'\setminus\overline\Omega$, and $\zeta(\sigma) = (\xi(\sigma), 1)$.
     

    
Then for any $N\in\mathbb N$ there exists a quasimode $u\in C^\infty_c(Z)$ of the form 
$$u(x,t)=\sum_{l=0}^{N-1} (-ik)^{-l}\int_0^T\bigg( e^{ik\varphi(x,t,\sigma)}a'_l(x,t,\sigma) + (-ik)^{-1} e^{ik\varphi(x,0,\sigma)}a''_l(x,t,\sigma)\bigg)d\sigma$$
such that 
\begin{equation} \label{pu_formula}
Pu = (-ik) (e^{ik\varphi(z,0)} a'(z,0)-e^{ik\varphi(z,T)} a'(z,T)) + r_N
\end{equation}
in $Z$, and $r_N \in L^2((0,T),L^2(\Omega'))$ satisfies
\[
\norm{r_N}_{L^2((0,T),L^2(\Omega'))} = O(k^{-N}) \text{ as $k \to \infty$}.
\]
The functions $\varphi, a'_l, a''_l \in C^{\infty}(Z\times [0,T])$ are smooth and independent of $k$, and respectively solve suitable eikonal, transport, and Volterra equations. In particular, the complex phase function $\varphi$ satisfies 
\[
\varphi(z(\sigma),\sigma) = 0, \qquad \nabla_z \varphi(z(\sigma), \sigma) = \zeta(\sigma), \qquad \p_{z_j z_k} \varphi(z(\sigma), \sigma) = H_{jk}(\sigma)
\]
where $H(\sigma)$ solves the matrix Riccati equation \eqref{matrix-riccati} with $\mathrm{Im}(H(\sigma))$ positive definite, and the principal amplitude $a_0'$ solves the transport equation
\[
-\p_{\sigma}(a_0'(z(\sigma),\sigma)) + b(z(\sigma),\sigma) a_0'(z(\sigma),\sigma) = 0
\]
where $b(z,\sigma) := 2Q\varphi(z,\sigma) + \frac{G(x,0)|\nabla\varphi(z,\sigma)|^2}{\p_t\varphi(z,\sigma)}$.
\end{Theorem}

The equation \eqref{pu_formula} states that $Pu$ is $O(k^{-N})$ away from the endpoints $z(0)$ and $z(T)$. In particular $Pu$ is an approximate solutions (with error $O(k^{-N})$) near the smaller cylinder $[0,T] \times \ol{\Omega}$. If there is no memory term, the $a''$ terms in Theorem \ref{GO-construction-general} are zero and $u$ resembles a standard Gaussian beam quasimode for the wave equation (see e.g.\ \cite{KKL01}). However, the presence of a memory term requires us to include the $a''$ amplitudes, and these are indeed obtained via solving Volterra equations. Moreover, as shown in Proposition \ref{geom-optics}, if the sound speed is known to be simple, an easier \emph{geometrical optics} construction with real phase functions is sufficient.

With the approximate solutions at hand, we can prove a propagation of singularities result for the semiclassical wave front set of solutions. This will make use of some concepts from semiclassical analysis, which we will rigorously define in Section 2 before the proof of Theorem \ref{propagation} (see also \cite{Zworski,OSSU}). In our case, the small semiclassical parameter $h$ will be $h:=k^{-1}$. The following propagation of singularities result will be used when solving the viscoacoustic inverse problem.

\begin{Theorem}\label{propagation}
    Let $u$ be an $L^2$-tempered solution of $Pu=0$ in $Z$. Consider the Hamiltonian flow $(x(\sigma),t_0+\sigma, \xi(\sigma),1)$ of the principal symbol of $P$ through $(x_0,t_0,\xi_0,1)$, where $x(\sigma)$ is a geodesic and $\xi(\sigma)$ is the parallel transport of $\xi_0$ along the geodesic. Then $(x(\sigma),t_0+\sigma, \xi(\sigma),1)\in WF_{scl}(u)$ if and only if $(x_0,t_0,\xi_0,1)\in WF_{scl}(u)$. 
\end{Theorem}

Note that the above theorem describes the singularities at points $(x,t,\xi,\tau)$ when $\tau \neq 0$. Theorem \ref{propagation} shows that the memory term does not affect the propagation of such singularities, and they behave exactly like those for the standard wave equation. This is due to the fact that the memory term $R$ is of lower order (with respect to the parameter $k$) than the wave term $Q$. Due to an integration by parts in time, we obtain $Ru = O(k)$, while of course $Qu = O(k^2)$. As mentioned above, solutions of the viscoacoustic equation can also have stationary singularities, which correspond to the case where $\tau=0$ (see Remark \ref{stationary-sing}). We will not study these singularities since they are not needed in the inverse problem.

Using propagating singularities as for the wave equation, we can recover the scattering relation and the travel time (the so called \emph{lens data}) from the source-to-solution data. This is the content of Proposition \ref{prop:from-cauchy-data-toscattering-relation}. The sound speed can then itself be recovered from the lens data under suitable geometric assumptions. In dimension $n=2$, it is a classical result by Muhometov \cite{Muh77}, see also \cite{Cro91, PSU23}, that a simple metric can be recovered within a given conformal class from the lens data. In dimension $n\geq 3$, Stefanov, Uhlmann and Vasy have obtained that the lens data suffices in order to recover a metric within a given conformal class, assuming that the manifold is foliated by strictly convex hypersurfaces \cite{SUV15}. Both the $n=2$ and the $n\geq 3$ cases further assume that the manifold $(\Omega, c(x)^{-2}g_0)$ is strictly convex, where $g_0$ represents the Euclidean metric. The use of the cited lens rigidity results, and related invertibility results for the geodesic X-ray transform, is the reason for the geometric assumptions (A1)-(A3) appearing in our main theorems. 

Finally, we are able to achieve unique recovery for the time derivatives of the memory kernel at time $t=0$ by means of Proposition \ref{formula}, which relates the values of the quasimode along a fixed null bicharacteristic with the geodesic X-ray transform of a quantity depending on the memory kernel. The derivatives of $G$ of all orders can be obtained by repeating this argument recursively. These results are contained in the following theorem:

 \begin{Theorem}\label{main}
     Let $\Omega \Subset\Omega'\subset\mathbb R^n$, $n\geq 2$, be open bounded sets with smooth boundary, and assume $-S' < 0 < T < T'$. Define the spacetime cylinders $M:= [0,T]\times\overline\Omega$ and $M':=[-S', T']\times\overline\Omega'$. Let $c_1, c_2 \in C^\infty(\mathbb R^n)$ be positive functions verifying $c_1=c_2$ in $\overline\Omega'\setminus\Omega$ and the geometric conditions (A1)-(A3). Moreover, let $G_1, G_2\in C^\infty(M)$ be memory kernels verifying $G_1=G_2$ in $\overline\Omega'\setminus\Omega$, and consider the source-to-solution data $S_1, S_2$ corresponding to $c_1, G_1$ and $c_2, G_2$ respectively. Assume that the source-to-solution data coincide, i.e.
    $$ \mathcal S_1(f) = \mathcal S_2(f), \qquad \mbox{for all } f\in C^\infty_c(M'\setminus M). $$
    Then $c_1=c_2$ and $\p^j_t G_1|_{t=0} = \p^j_t G_2|_{t=0}$ for all $j\in\mathbb N$.
\end{Theorem}

Indeed we prove that the time derivatives of all orders of the memory kernel at time $t=0$ and the sound speed are uniquely determined by the source-to-solution map. If $n=2$ this is analogous to the result in \cite{Rom-article} (which considered the viscoelastic case), but for $n \geq 3$ there are many non-simple geometries satisfying (A3) \cite{PSUZ19}. 

As an application of our results, in the last Section \ref{sec:emm} we present the so called extended Maxwell model for viscoacoustics, for which we show uniqueness in the recovery of the associated viscous and elastic parameters $\alpha,\beta$:

\begin{Theorem}\label{th-recovering-EMM}
    Let $\widetilde{G}_1(x,t)$ and $\widetilde{G}_2(x,t)$ have the form given in Section \ref{sec:emm}, and assume that the time derivatives of $\widetilde G_1(x,t)$ and $\widetilde G_2(x,t)$ at $t=0$ coincide for all orders $k\in\{0,...,2N-1\}$. Then $\widetilde{G}_1(x,t) = \widetilde{G}_2(x,t)$.
\end{Theorem}

The argument uses our main result of Theorem \ref{main}, as well as a particular algebraic manipulation via Vandermonde matrices. We expect that our methods would apply also to other parametric models.

\subsection{Motivation and connection to the literature} 

The viscoacoustic wave equation has been an object of study since at least the work of Dafermos \cite{Dafermos}, who has considered the direct problem and analyzed the behavior of its solutions for times $t\rightarrow\infty$ by means of the theory of abstract Volterra equations. Carcione, Kosloff and Kosloff \cite{CKK} have applied the viscoacoustic wave equation in order to model the anelastic properties of real materials in seismology, using the theory of viscoelasticity based on Boltzmann's superposition principle. A related inverse problem was studied in \cite{Leander} by Leander, who has analyzed the transient linear acoustic wave propagation in fluids with Maxwell relaxation processes, in particular in the case of a single process, and has obtained estimates of the governing parameters of the Maxwell fluid by means of numerical approximation methods. The fundamental work of Christensen \cite{Christensen} on the theory of viscoelasticity has put the bases for the results obtained in more recent years. The numerical methods of Bai, Yingst, Bloor and Leveille \cite{BYBL} have shown that viscoacoustic inversion is a more realistic approach than acoustic inversion in the study of elasticity of real earth materials. Hao and Alkhalifah \cite{HA} have developed a quite general
representation of the viscoacoustic wave
equation in the cases of (quasi-elastic) orthorhombic anisotropy and transverse isotropy, and have studied the related direct problem. In relatively recent work \cite{HG}, Hao and  Greenhalgh have succeeded in incorporating the quality factor $Q$, which is a common phenomenological description of seismic wave attenuation in the interior of the Earth, into a viscoacoustic wave equation. In seismic applications the relevant case is $n=3$, and caustic formation is essentially unavoidable when considering the case of \textit{P} waves.

The wave equation, which corresponds to the viscoacoustic wave equation in the case of vanishing memory kernel, plays a central role in seismic modeling, processing, imaging and inversion.
Inverse problems for the wave equation have been studied in several formulations and through a variety of methods. Rakesh and Symes \cite{RS88} have used geometrical optics solutions for recovering a potential in the wave equation with fixed sound speed equal to $1$. The method of geometrical optics has been used also in several works by Stefanov and Uhlmann, such as \cite{SU05} for the case of simple metrics, and in collaboration with Oksanen and Salo for the more general case of principal type operators \cite{OSSU}. The boundary control method has also been very successful in the study of inverse problems for the wave equation: for this we refer to the book by Katchalov, Kurylev and Lassas \cite{KKL01}, as well as to the references therein. Finally, we report the technique based on distorted plane waves that appeared in a recent work by Oksanen, Rakesh and Salo \cite{ORS24}.  

The memory term appearing in the viscoacoustic operator $P$ defined in \eqref{viscoacoustic-equation} has the same order as the wave operator. Related inverse problems with lower order memory terms have been investigated in works of Bukhgeim, Dyatlov, Isakov and Uhlmann, e.g. \cite{BDI97, BD00, BDU07} where the authors considered the recovery of the memory term from Dirichlet-to-Neumann data for the operator
$$ u_{tt} - \Delta u +\int_{0}^t k(x,t-s)u(x,s)ds $$
and for its version with first and zero-th order perturbations. For the actual viscoelastic case where the memory term involves second order derivatives in space, \cite{Rom-article, Romanov} and references therein state several results on inverse problems for integro-differential equations with convolutional memory terms for both electrodynamics and viscoelasticity. 

The viscoacoustic wave equation and related inverse problems constitute a very active field of study. To our knowledge, propagation of singularities in the case of $C^{1,1}$ coefficients is currently being researched by Mazzuccato and Qi, who are employing a Hart Smith type parametrix \cite{HartSmith} and microlocal factorization techniques. However, both the methods used and the results obtained, while related to our study, do not substantially overlap with it.

\subsection{Organization of the article} 
 The rest of the article is organized as follows. In Section \ref{sec_propagation} we prove Theorem \ref{GO-construction-general} and show how to construct approximate solutions, both in the form of geometrical optics in the case of simple metrics, and more generally as Gaussian beam quasimodes when the metric is not known to be simple a priori. Our propagation of singularities result of Theorem \ref{propagation} is also proved in the same section. In Section \ref{sec_wellposedness} we briefly study the well-posedness of the direct problem for the viscoacoustic wave equation, and are thus able to define the source-to-solution map, which encodes the data. The solution of the inverse problem itself is presented in Section \ref{sec_inverse}, which includes the proof of our main result, Theorem \ref{main}. Section \ref{sec:emm} is dedicated to the reconstruction of the viscosity and elasticity coefficients in the extended Maxwell model, using our Theorem \ref{main} as the main tool. Finally, we have added a short Appendix in which we present the derivation of the viscoacoustic operator from the classical theory of viscoelasticity discussed in \cite{Christensen}.

\subsection*{Acknowledgments} 
G. Covi and M. Salo are partially supported by the FAME flagship of the Research Council of Finland (grant 359186). M.V. de Hoop gratefully acknowledges the support of the Simons Foundation under the MATH + X program, the National Science Foundation under grant DMS-2407456, and the corporate members of the Geo-Mathematical Imaging Group at Rice University. 

\section{Existence of approximate solutions and propagation of singularities} \label{sec_propagation}

In this section we construct special approximate solutions to the viscoacoustic equation. We start from the case of simple sound speed $c$, for which we show how to construct approximate geometrical optics solutions.
\begin{Proposition}\label{geom-optics}
    Suppose that the sound speed $c$ is simple in $\ol{\Omega}$, and assume that $\varphi(x,t):= t-\psi(x)$, where $\psi \in C^{\infty}(\ol{\Omega})$ solves the eikonal equation 
\[
|\nabla \psi| = \frac{1}{\sqrt{c(x)}} \quad\text{ in $\Omega$}.
\]
Further, assume that $a'_0 \in C^{\infty}(\ol{\Omega} \times [0,T])$ solves the transport equation 
\[
2 (\p_t + c(x) \nabla \psi \cdot \nabla) a'_0 + b(x) a'_0 = 0 \quad\text{ in $\Omega\times [0,T]$} 
\]
where $b(x) := (c^{-1}G)(x,0) - Q\psi(x)$.

Then for any $N\in\mathbb N$ there exists a function $u\in C^\infty(\ol{\Omega} \times [0,T])$ of the form 
\[
u(x,t) =  e^{ik \varphi(x,t)} \left(a'_0(x,t) + \sum_{l=1}^{N-1}(-ik)^{-l}\bigg( a'_l(x,t)  + e^{-ikt}a''_{l-1}(x,t)\bigg)\right)
\]
such that 
\[
\norm{Pu}_{ L^2((0,T), L^2(\Omega))} = O(k^{-N}).
\]
The functions $a'_j, a''_j \in C^{\infty}(\ol\Omega\times [0,T])$ are independent of $k$, and respectively solve transport and Volterra equations.
\end{Proposition}

\begin{proof}
\textbf{Step 1 (Set-up).}
Let $u(x,t) := e^{ik \varphi(x,t)} a(x,t)$, where $\varphi(x,t):=t-\psi(x)$ for some real valued function $\psi$ to be determined, and $$a(x,t):=\sum_{l=0}^{N-1}(-ik)^{-l}a_l(x,t), \qquad\mbox{with}\qquad a_l(x,t):= a'_l(x,t) + (-ik)^{-1}e^{-ikt}a''_l(x,t).$$
The wave operator part satisfies 
\[
Qu = e^{ik\varphi} (k^2 q(x,t,\nabla_{x,t} \varphi) a + ik La  + Qa),
\]
where $q(x,t,(\xi,\tau)) := \frac 12(\tau^2 - c(x) |\xi|^2)$ and $L := -(\p_t \varphi \p_t - c \nabla \varphi \cdot \nabla) + Q\varphi$. For the memory term we compute 

\begin{align*}
Ru &= -\frac 12\int_0^t  e^{ik\varphi(s)} (\nabla+ik\nabla \varphi(s)) \cdot [ G(t-s) (\nabla+ik\nabla \varphi(s)) a(s) ] \,ds \\
 &=  -\frac 12 \left(-k^2 \int_0^t e^{ik\varphi(s)} G(t-s) |\nabla \varphi(s)|^2 a(s) \,ds\right. \\
 & \qquad + ik \int_0^t e^{ik\varphi(s)} \left[ 2 \nabla \varphi(s) \cdot G(t-s) \nabla a(s) + \nabla \cdot (G(t-s) \nabla \varphi(s) a(s)) \right] \,ds \\
 & \qquad \left.+ \int_0^t e^{ik\varphi(s)}  \nabla \cdot (G(t-s) \nabla  a(s) ) \,ds\right)
 \\ & =:
 -k^2 \int_0^t e^{ik\varphi(s)} A(t,s) \,ds + ik \int_0^t e^{ik\varphi(s)} B(t,s) \,ds + \int_0^t e^{ik\varphi(s)} C(t,s) \,ds .
\end{align*}

In order to study the behaviour in $k$ of the memory term, we can integrate by parts using the formula
\begin{equation}\label{int-by-parts}\begin{split}
\int_0^t e^{ik\varphi(s)} h(s) \,ds &= \int_0^t \p_s(e^{ik\varphi(s)}) \frac{h(s)}{ik \p_t \varphi(s)} \,ds \\ & = \frac{1}{ik} \left[ \frac{e^{ik\varphi(s)} h(s)}{\p_t \varphi(s)} \right]_{s=0}^t -\frac 1{ik} \int_0^t e^{ik\varphi(s)} \p_s \left(\frac{h(s)}{\p_t \varphi(s)} \right) \,ds.
\end{split}\end{equation}
We define by recursion $h^{(j+1)} := \p_s\left(\frac{h^{(j)}(s)}{\p_t\varphi(s)}\right)$ for all non-negative integers $j$, and $h^{(0)}:= h$. Then we can iterate the previous formula to obtain, for all non-negative integers $N$,
\begin{equation}\label{int-by-parts-iterated}\begin{split}
\int_0^t e^{ik\varphi(s)} h^{(0)}(s) \,ds = -\sum_{j=1}^{N} (-ik)^{-j} \left[ \frac{e^{ik\varphi(s)} h^{(j-1)}(s)}{\p_t \varphi(s)} \right]_{s=0}^t + (-ik)^{-N}\int_0^t e^{ik\varphi(s)} h^{(N)}(s) \,ds.
\end{split}\end{equation}

\textbf{Step 2 (Eikonal equation).} Observe that we have 
$$\p_t \varphi \equiv 1, \qquad \nabla\varphi(x,t) = -\nabla\psi(x), \qquad \varphi(x,s)-\varphi(x,t) = s-t, \qquad L = -( \p_t + c \nabla \psi \cdot \nabla) -\frac 12 \nabla \cdot (c \nabla \psi). $$ Moreover, because $\psi$ is a real function we have $$|e^{ik\varphi(s)}| = |e^{iks}| = 1,$$ and so integrating by parts $N\in\mathbb N$ times we get
\begin{equation}\label{eq:int-by-parts-iterated-2}
\int_0^t e^{iks} h(t,s) \,ds = - \sum_{j=1}^{N-1} (-ik)^{-j}\left[ e^{iks}\p^{j-1}_s h(t,s) \right]_{s=0}^t + O(k^{-N}).
\end{equation}
Applying this formula to the computation of $Ru$, it becomes apparent that the memory term $e^{-ik\varphi(t)}Ru$ has order at most $O(k)$. Therefore, it holds that
$$ e^{-ik\varphi(t)}Pu = k^2 q(x,t,\nabla_{x,t} \varphi) a + O(k).$$

Since the sound speed is simple, we can solve the eikonal equation $(\p_t \varphi)^2 - c |\nabla \varphi|^2 = 0$ by letting $\varphi(x,t) = t - \psi(x)$ where $\psi$ was a solution of $|\nabla \psi(x)|^2 = \frac{1}{c(x)}$.

\textbf{Step 3 (Lower order terms).} Given that $A(x,t,s) = -\frac 12(c^{-1}G)(x,t-s) a(x,s)$, we will also write
{ 
\begin{align*}
    A_l(x,t,s) &:=-\frac 12 (c^{-1}G)(x,t-s) a_l(x,s) = -\frac 12(c^{-1}G)(x,t-s) \bigg(a'_l(x,s) + (-ik)^{-1}e^{-iks}a''_l(x,s)\bigg)
    \\ & =: A'_l(x,t,s) + (-ik)^{-1}e^{-iks}A''_l(x,t,s),
\end{align*} 
}
and similarly for $B(x,t,s)$ and $C(x,t,s)$, for all $l\in\{0,...,N-1\}$. Then we compute
\begin{align*}
    La & = \sum_{l=0}^{N-1}(-ik)^{-l}\bigg(La'_l + e^{-ikt}\bigg[  2  a''_l + (-ik)^{-1}La''_l \bigg]\bigg),
\end{align*}
\begin{align*}
    Qa & = \sum_{l=0}^{N-1}(-ik)^{-l}\bigg(Qa'_l + e^{-ikt}\bigg[(-ik)a''_l +2\p_ta''_l +(-ik)^{-1} Q a''_l\bigg]\bigg),
\end{align*}
and by the eikonal equation the operator satisfies
\begin{align*}
    e^{-ik\varphi(t)}Pu & = e^{-ik\varphi(t)}(Qu + Ru)
    \\ & =
    ik La  + Qa  + \int_0^t e^{ik(s-t)} \bigg(-k^2 A + ik B + C\bigg)(t,s)\,ds
    \\ & =
    ik La  + Qa  + \int_0^t e^{ik(s-t)}  \sum_{l=0}^{N-1}\bigg((-ik)^{2-l}A'_l+ (-ik)^{-l+1}e^{-iks}A''_l\bigg)(t,s)\,ds \\ & \qquad + 
    \int_0^t e^{ik(s-t)}\sum_{l=0}^{N-1}\bigg(-(-ik)^{1-l}B'_l - (-ik)^{-l}e^{-iks}B''_l\bigg)(t,s) \,ds 
    \\ & \qquad + \int_0^t e^{ik(s-t)}\sum_{l=0}^{N-1}\bigg((-ik)^{-l}C'_l + (-ik)^{-l-1}e^{-iks}C''_l\bigg)(t,s) \,ds
    \\ & =
    ik La  + Qa  + \sum_{l=0}^{N-1}(-ik)^{-l} \int_0^t e^{ik(s-t)}\bigg(-k^2A'_l + ikB'_l+ C'_l\bigg)(t,s) \,ds 
    \\ & \qquad + e^{-ikt}\sum_{l=0}^{N-1}(-ik)^{-l-1}
    \int_0^t  \bigg( -k^{2}A''_l+ikB''_l + C''_l\bigg)(t,s) \,ds
\end{align*}
For the first integral we use the integration by parts formula \eqref{eq:int-by-parts-iterated-2}, which gives
\begin{align*}
    &  \sum_{l=0}^{N-1}(-ik)^{-l} \int_0^t e^{ik(s-t)}\bigg(-k^2A'_l(t,s) + ikB'_l(t,s) + C'_l(t,s)\bigg) \,ds =
    \\ & = - \sum_{l=0}^{N-1}\sum_{j=1}^{N-1}  (-ik)^{-(l+j)}\left[ e^{ik(s-t)} \p^{j-1}_s\bigg(-k^2A'_l + ikB'_l + C'_l\bigg)(t,s) \right]_{s=0}^t + O(k^{-N})
    \\ & = - \sum_{\sigma=1}^{N-1}(-ik)^{-\sigma}\sum_{l=0}^{\sigma-1}  \left[ e^{ik(s-t)} \p^{\sigma-l-1}_s\bigg(-k^2A'_{l}+ ikB'_{l} + C'_{l}\bigg)(t,s) \right]_{s=0}^t + O(k^{-N})
\end{align*}
Then we can write
\begin{align*}
    & e^{-ik\varphi(t)}Pu  = 
    - \sum_{l=0}^{N-1}(-ik)^{1-l}La'_l  - e^{-ikt}\sum_{l=0}^{N-1}(-ik)^{-l} \bigg(  2(-ik)  a''_l + La''_l \bigg)  
    \\ & \quad + \sum_{l=0}^{N-1}(-ik)^{-l}Qa'_l  + e^{-ikt}\sum_{l=0}^{N-1}(-ik)^{-l-1}\bigg((-ik)^{2}a''_l +2(-ik)\p_ta''_l + Q a''_l\bigg)
    \\ & \quad - \sum_{l=1}^{{ N}-1}(-ik)^{-l}\sum_{j=0}^{l-1} \p^{l-j-1}_s\bigg(     (-ik)^2A'_{j}- (-ik)B'_{j} + C'_{j}\bigg)(t,t)
    \\ & \quad + e^{-ikt}\sum_{l=1}^{{ N}-1}(-ik)^{-l}\sum_{j=0}^{l-1}  \p^{l-j-1}_s\bigg((-ik)^2A'_{j}- (-ik)B'_{j} + C'_{j}\bigg)(t,0) 
    \\ & \quad + e^{-ikt}\sum_{l=0}^{N-1}(-ik)^{-l-1}
    \int_0^t  \bigg( (-ik)^2A''_l - (-ik)B''_l + C''_l\bigg)(t,s) \,ds + O(k^{-N})
    \\ & = 
    - (-ik)\big\{La'_0\big\} - \big\{La'_1 - Qa'_0\big\}
    \\ & \quad - \sum_{l=1}^{N-1}(-ik)^{-l}\bigg\{ La'_{l+1} -Qa'_l +\sum_{j=0}^{l-1} \p^{l-j-1}_s\bigg(     (-ik)^2A'_{j}- (-ik)B'_{j} + C'_{j}\bigg)(t,t)\bigg\}
    \\ & \quad - e^{-ikt}\sum_{l=0}^{N-1}(-ik)^{-l+1} \bigg( a''_l - \int_0^t   A''_l (t,s) \,ds\bigg) 
    \\ & \quad - e^{-ikt}\sum_{l=0}^{N-1}(-ik)^{-l}\bigg( La''_l -2\p_ta''_l + \int_0^t   B''_l (t,s) \,ds\bigg) 
    \\ & \quad + e^{-ikt}\sum_{l=0}^{N-1}(-ik)^{-l-1} \bigg( Q a''_l + \int_0^t   C''_l(t,s) \,ds\bigg)
    \\ & \quad + e^{-ikt}\sum_{l=1}^{N-1}(-ik)^{2-l}\sum_{j=0}^{l-1}  \big(\p^{l-j-1}_sA'_{j}\big)(t,0)   - e^{-ikt}\sum_{l=1}^{N-1}(-ik)^{1-l}\sum_{j=0}^{l-1}  \big(\p^{l-j-1}_sB'_{j} \big)(t,0)    
    \\ & \quad + e^{-ikt}\sum_{l=1}^{N-1}(-ik)^{-l}\sum_{j=0}^{l-1}  \big(\p^{l-j-1}_sC'_{j}\big)(t,0) + O(k^{-N}),
\end{align*}
and after some reordering of the terms 
\begin{align*}
    & e^{-ik\varphi(t)}Pu  = 
    - (-ik)\big\{La'_0\big\} - \big\{La'_1 - Qa'_0\big\}
    \\ & \quad - \sum_{l=1}^{N-1}(-ik)^{-l}\bigg\{ La'_{l+1} -Qa'_l +\sum_{j=0}^{l-1} \p^{l-j-1}_s\bigg(     (-ik)^2A'_{j}- (-ik)B'_{j} + C'_{j}\bigg)(t,t)\bigg\}
    \\ & \quad - e^{-ikt}(-ik) \bigg\{ a''_0 - \int_0^t   A''_0 (t,s) \,ds\bigg\}
    \\ & \quad - e^{-ikt} \bigg\{ a''_1 - \int_0^t   A''_1 (t,s) \,ds +  La''_0 -2\p_ta''_0 + \int_0^t   B''_0 (t,s) \,ds\bigg\}
    \\ & \quad - e^{-ikt}\sum_{l=0}^{N-1}(-ik)^{-l-1} \bigg\{ a''_{l+2} - \int_0^t   A''_{l+2} (t,s) \,ds+ La''_{l+1} -2\p_ta''_{l+1}
    \\ & \qquad\qquad\qquad\qquad\qquad\quad + \int_0^t   B''_{l+1} (t,s) \,ds -Q a''_l - \int_0^t   C''_l(t,s) \,ds\bigg\}
    \\ & \quad + e^{-ikt}(-ik)A'_{0}(t,0)  + e^{-ikt}\bigg(\sum_{j=0}^{1}\p^{1-j}_sA'_{j} - B'_{j}\bigg)(t,0)    
    \\ & \quad + e^{-ikt}\sum_{l=1}^{N-1}(-ik)^{-l}\bigg(\sum_{j=0}^{l+1}  \p^{l+1-j}_sA'_{j}  -\sum_{j=0}^{l}  \p^{l-j}_sB'_{j} + \sum_{j=0}^{l-1}  \p^{l-j-1}_sC'_{j}\bigg)(t,0)
    \\ & \quad + O(k^{-N}).
\end{align*}
If we define the operator $V$ such that 
{ 
$$ Vu(x,t) := u(x,t) +\frac 12 \int_0^t (c^{-1}G)(x,t-s)u(x,s) ds, $$}
then we obtain
\begin{align*}
    & e^{-ik\varphi(t)}Pu  = 
    - (-ik)\big\{La'_0+  A'_{0}(t,t)\big\} - \bigg\{La'_1 + A'_{1}(t,t) - Qa'_0 + \bigg(\p_s  A'_{0}-  B'_{0}\bigg)(t,t) \bigg\}
    \\ & \quad - \sum_{l=1}^{N-1}(-ik)^{-l}\bigg\{ La'_{l+1} + A'_{l+1}(t,t) -Qa'_l + \bigg(\sum_{j=0}^{l} \p^{l+1-j}_s  A'_{j}  -\sum_{j=0}^{l} \p^{l-j}_s B'_{j} +  \sum_{j=0}^{l-1} \p^{l-j-1}_s C'_{j}\bigg)(t,t)\bigg\}
    \\ & \quad - e^{-ikt}(-ik) \big\{V a''_0 - A'_{0}(t,0)\big\} 
    \\ & \quad - e^{-ikt} \bigg\{ V a''_1 +  La''_0 -2\p_ta''_0 + \int_0^t   B''_0 (t,s) \,ds - \bigg(\sum_{j=0}^{1}\p^{1-j}_sA'_{j} - B'_{j}\bigg)(t,0)\bigg\}
    \\ & \quad - e^{-ikt}\sum_{l=1}^{N-1}(-ik)^{-l} \bigg\{ Va''_{l+1} + La''_{l} -2\p_ta''_{l}+ \int_0^t   B''_{l} (t,s) \,ds -Q a''_{l-1} - \int_0^t   C''_{l-1}(t,s) \,ds    
    \\ & \qquad\qquad\qquad\qquad\qquad -\bigg(\sum_{j=0}^{l+1}  \p^{l+1-j}_sA'_{j}  -\sum_{j=0}^{l}  \p^{l-j}_sB'_{j} + \sum_{j=0}^{l-1}  \p^{l-j-1}_sC'_{j}\bigg)(t,0) \bigg\}
    \\ & \quad + O(k^{-N}).
\end{align*}

This implies that the equations solved by the functions $a'_l, a''_l$ are of the form
\begin{equation}\label{eq:transport-volterra}
    La'_l + A'_{l}(t,t) = \mathcal R'_l(x,t), \qquad Va''_l = \mathcal R''_l(x,t),
\end{equation} 
for all $l\in\{0,...,N-1\}$, where the right-hand side terms $\mathcal R'_l$ depend only on $\{a'_j\}_{j<l}$, and the right-hand side terms $\mathcal R''_l$ depend only on $\{a''_l\}_{j<l}$ and $\{a'_j\}_{j\leq l}$. This means that we can solve the equations recursively, starting from $a'_0$ and then alternating $a''_0, a'_1, a''_1, a'_2, ...$ . 

\textbf{Step 4 (Volterra equations).} For the second equation in \eqref{eq:transport-volterra}, which is a \emph{Volterra integral equation of the second kind}, we refer to \cite[Theorems 1.2.3 and 2.1.1]{Brunner}. The equation takes the form
$$ u(t) = f(t) + \int_0^t K(t-s)u(s)ds, $$
with $K:=c^{-1}G$. Define $K_1 := K$,
$$ K_n(t) := \int_0^t K_1(t-s)K_{n-1}(s)ds, \qquad n\geq 2. $$
and finally the resolvent kernel $R$ as the (absolutely and uniformly) convergent Neumann series
$$  R(t) := \sum_{n=1}^\infty K_n(t). $$
The unique solution $u(t)$ of the Volterra integral equation then takes the form
$$ u(t) = f(t) + \int_0^t R(t-s)f(s)ds, $$
which in particular implies that for all $N\in\mathbb N$ we have the estimate
$ \|u\|_{C^N} \lesssim \|f\|_{C^N}(1+\|R\|_{C^N})$.

\textbf{Step 5 (Transport equations).} We now solve the first equation in \eqref{eq:transport-volterra}. Define the operators $T,Z$ and the function $g$ as
{ $$ g(x):= -\frac 12(c^{-1}G)(x,0), \qquad Z:= -\frac 12 (\p_t  + c(x) \nabla \psi(x) \cdot \nabla),\qquad T:= -Z+\frac12 (Q\psi(x)-g(x)).$$}
Then for every $l \in\{0,...,N-1\}$ we have
\begin{align*}
    Ta'_l(x,t) = La'_l(x,t) +  A'_l(x,t,t) = \mathcal R'_l(x,t). 
\end{align*} 
To solve the transport equation, assume that $\mu(x,t) = \mu_G(x,t)$ is an integrating factor solving
\[
\mu^{-1} Z\mu = \frac{g-Q\psi}{2},
\]

\noindent so that
\begin{equation*}
    \begin{split}
        Z(\mu a'_l) =  \mu \left( Z + (\mu^{-1}Z \mu)  \right)a'_l = \frac{\mu\, Ta'_l}{2} = \frac{\mu\, \mathcal R'_l}{2}.
    \end{split}
\end{equation*}
The integral curve $\eta_x(t)$ of $Z$ with $\eta_x(0) = (x,0)$ is given by $\eta_x(t) = (\alpha_x(t),t)$, where $\alpha_x(t)$ satisfies $\dot{\alpha}_x(t) = (c \nabla \psi)(\alpha_x(t))$ and $\alpha_x(0) = x$. One can check that $\alpha_x(t) = \gamma_{x,(c\nabla \psi)(x)}(t)$ where $\gamma_{x,v}$ is the geodesic of the metric $\frac{1}{c(x)} \delta_{jk}$ with $\gamma_{x,v}(0) = x$ and $\dot{\gamma}_{x,v}(0) = v$. We can choose $\mu$ as 
\[
\mu(\alpha_x(t), t) = \exp \left( \frac{1}{2}\int_0^t \big( g-Q\psi\big)(\alpha_x(s)) \,ds \right),
\]
which can be rewritten as 
\[
\mu(x, t) =  \exp \left( \frac{1}{2}\int_0^t \big( g-Q\psi\big)(\alpha_x(-s)) \,ds \right).
\]
 We now fix an initial value $a'_l(x,0) = f'_l(x)$, and want to find $a'_l(x,t)$ solving 
\[
Z(\mu a'_l)  = \frac{\mu(x,t)\mathcal R'_l(x,t)}{2}, \qquad a'_l(x,0) = f'_l(x).
\]
We need that 
\[
(\mu a'_l)(\alpha_x(t),t) = (\mu a'_l)(x,0) + \frac 12\int_0^t \mu(\alpha_x(s),s)\mathcal R'_l(\alpha_x(s),s) \,ds.
\]
Since $\mu(x,0) = 1$, the solution of the transport equation is 
\begin{align*}
    a'_l(\alpha_x(t),t) & = \frac{1}{\mu(\alpha_x(t),t)} \left(f'_l(x) + \frac 12\int_0^t \mu(\alpha_x(s),s)\mathcal R'_l(\alpha_x(s),s) \,ds\right),
\end{align*}
which can be rewritten as
\begin{align*}
    a'_l(x,t) & =  \frac{f'_l(\alpha_x(-t))}{\mu(x,t)}  + \frac 1{2\mu(x,t)}\int_0^t \big(\mu \mathcal R'_l\big)(\alpha_x(-s),t-s) \,ds .
\end{align*}
Assuming that $f_{l+1}\equiv 0$ for all $l\geq 0$, we get
\begin{align}\label{eq:face-of-the-amplitudes}
    a'_0(x,t) & =  \frac{f'_0(\alpha_x(-t))}{\mu(x,t)}, \qquad\mbox{and}\qquad
    a'_{l+1}(x,t) & =  \frac 1{2\mu(x,t)}\int_0^t \big(\mu \mathcal R'_{l+1}\big)(\alpha_x(-s),t-s) \,ds, \quad l\geq 0,
\end{align}
where for all $l\geq 0$ it holds
\begin{equation}\label{eq:remainders-from-transport} \mathcal R'_{l+1}(x,t) = -Qa'_l + \p_s  A'_{l}(t,t) - B'_{l}(t,t) + \sum_{j=0}^{l-1}  \p^{l-j-1}_s\bigg(\p^{2}_s  A'_{j}  -\p_s B'_{j} + C'_{j}\bigg)(t,t). \end{equation}

Observe that by the definition of the amplitude $a(x,t)$ it holds
\begin{equation}\label{eq:amplitude-powers-k}a(x,t)= a'_0(x,t) + \sum_{l=0}^{N-1}(-ik)^{-l-1}\bigg( a'_{l+1}(x,t) + e^{-ikt}a''_{l}(x,t)\bigg) = \frac{f'_0(\alpha_x(-t))}{\mu(x,t)} + O(k^{-1}).\end{equation}

\noindent If the initial value $f'_0$ is supported near $x_0$, then by the above formula we see that $a$ is itself supported near the curve $(\alpha_{x_0}(t),t)$ modulo an $O(k^{-1})$ term. This proves that there exist approximate geometrical optics solutions concentrating near geodesics of the metric $\frac{1}{c(x)} \delta_{jk}$, just like in the case of the standard wave equation. 

{  \textbf{Step 6 (Final estimate).} We have proved that $|Pu| = |e^{-ik\varphi(t)}Pu| = O(k^{-N})$. Therefore, the $L^2$ estimate $\norm{Pu}_{ L^2((0,T), L^2(\Omega))} = O(k^{-N})$
follows immediately by the boundedness of $\Omega\times(0,T)$.
}\end{proof}

Next, we turn to the case of general (thus, not necessarily simple) sound speed, for which we construct approximate solutions in the form of quasimodes. 

\begin{proof}[Proof of Theorem \ref{GO-construction-general}] \textbf{Step 1 (Set-up).} In this case we look for a Gaussian beam approximate solution of the form
$$u(x,t)=\int_0^T e^{ik\varphi(x,t,\sigma)}a'(x,t,\sigma)d\sigma + \int_0^T e^{ik\varphi(x,0,\sigma)}a''(x,t,\sigma)d\sigma,$$ where $a', a'',\varphi$ are allowed to be complex valued. Here we use the ansatz in \cite{OSSU} written in terms of superpositions of wave packets, which will be useful later when we prove propagation of singularities results. We also assume that 
$$a'(z,\sigma) = \sum_{l=0}^{N-1} (-ik)^{-l}a'_l(z,\sigma), \qquad a''(z,\sigma) = (-ik)^{-1}\sum_{l=0}^{N-1} (-ik)^{-l}a''_l(z,\sigma),$$
which gives
$$u(x,t)=\sum_{l=0}^{N-1} (-ik)^{-l}\int_0^T\bigg( e^{ik\varphi(x,t,\sigma)}a'_l(x,t,\sigma) + (-ik)^{-1} e^{ik\varphi(x,0,\sigma)}a''_l(x,t,\sigma)\bigg)d\sigma.$$

In what follows we allow $z=(x,t)$ and $\zeta=(\xi,\tau)$, while $\sigma\in [0,T]$ will be the parameter of the bicharacteristic $\gamma(\sigma) = (z(\sigma),\zeta(\sigma))$. We make this choice of symbols in order to avoid the confusion arising from the fact that the coordinates defining $P$ already include time $t$, and thus the space projection $z$ of $\gamma$ is actually in spacetime.

\textbf{Step 2 (Matrix Riccati equation).} We shall look for a phase function $\varphi(z,\sigma)$ satisfying 
\begin{equation}\label{infinite-order}
    q(\cdot, \nabla_z\varphi(\cdot,\sigma)) + \p_\sigma\varphi(\cdot,\sigma)=0 \quad \mbox{ to infinite order at } z(\sigma) \mbox{ for } \sigma\in[0,T], 
\end{equation}
\begin{equation}\label{other-conditions}
    \varphi(z(\sigma),\sigma)=0, \qquad \nabla_z\varphi(z(\sigma),\sigma) =\zeta(\sigma),
\end{equation} 
where we recall that 
{ 
$$q(z,\zeta) := \frac 12(\tau^2 - c(x)|\xi|^2).$$}
For the construction of $\varphi$ we propose
\begin{equation}\label{eq:define-varphi} \varphi(z,\sigma) = \zeta(\sigma)\cdot(z-z(\sigma)) +\frac 12 H(\sigma)(z-z(\sigma))\cdot(z-z(\sigma)) + \varphi_3(z,\sigma), \end{equation}
with $H$ complex and symmetric, and $\varphi_3$ vanishing of order $3$ at $z(\sigma)$. It is immediately seen that the above $\varphi$ satisfies \eqref{other-conditions}. As in \cite{OSSU}, using these and the Hamilton equations, one proves that \eqref{infinite-order} holds to first order. In order to show that \eqref{infinite-order} holds to second order, one has to solve the matrix Riccati equation \begin{equation}\label{matrix-riccati} \p_t H + D + (BH+HB^T) + HCH =0 ,\end{equation}
where $B,C$ and $D$ are matrices with components given by the second derivatives of the Hamiltonian $q(z,\zeta)$:
$$ B_{jl} := \frac{\p^2q}{\p z_l\p \zeta_j}, \qquad C_{jl} := \frac{\p^2q}{\p \zeta_j\p \zeta_l}, \qquad D_{jl} := \frac{\p^2q}{\p z_j\p z_l}. $$
By \cite[Lemma 2.56]{KKL01} we can find a solution with the additional property that $\Im H(\sigma)$ is positive definite for all $\sigma\in[0,T]$. The construction of $\varphi_3(z, \sigma)$ is completed by successively solving linear systems of ODEs, see \cite[Lemma 6.3]{MSS}.

We observe that we can ensure that $\Im\varphi(z,\sigma)>0$ when $a'$ and $a''$ are chosen to be supported sufficiently close to $\gamma$. In fact, because the first order term of $\varphi$ is real valued, we have
$$ \Im\varphi(z,\sigma) = \frac 12 \Im H(\sigma)(z-z(\sigma))\cdot(z-z(\sigma)) + O(|z-z(\sigma)|^3)>0 $$
close to $\gamma$ by the positive definitedness of $\Im H(\sigma)$. In particular, this implies that for all $s\in [0,t]$ and all $k\in\mathbb N$ it holds
\begin{align*}
    |e^{ik\varphi(x,s,\sigma)}| = |e^{ik\Re\varphi(x,s,\sigma)}e^{-k\Im\varphi(x,s,\sigma)}| = e^{-k\Im\varphi(x,s,\sigma)}\leq 1,
\end{align*}
which by the integration by parts formula \eqref{int-by-parts-iterated} ensures that the memory terms of $P$ contribute at most to order $O(k)$.

\textbf{Step 3 (Lower order terms).} By linearity, in order to compute $Pu$ it suffices to find
$$ P(e^{ik\varphi(x,t)}a'_l(x,t)) \qquad \mbox{and} \qquad P(e^{ik\varphi(x,0)}a''_l(x,t)), $$
where for the sake of simplicity we have suppressed the dependence on $\sigma$. By the computations in the proof of Proposition \ref{geom-optics}, we find
\begin{align*}
    P(e^{ik\varphi(x,t)}a'_l(x,t)) & =
    -e^{ik\varphi}( (-ik)^2 q(z,\nabla_z\varphi) + (-ik) L - Q)a'_l 
    \\& \quad 
    + (-ik)^2 \int_0^t e^{ik\varphi(s)} A'_l(t,s) \,ds - (-ik) \int_0^t e^{ik\varphi(s)} B'_l(t,s) \,ds + \int_0^t e^{ik\varphi(s)} C'_l(t,s) \,ds
\end{align*}
with clear meaning of the symbols $A'_l, B'_l, C'_l$. The integration by parts formula \eqref{int-by-parts-iterated} then gives 
\begin{align*}
    & P(e^{ik\varphi(x,t)}a'_l(x,t))  =
                           \\ & =
     -\frac{e^{ik\varphi}}{\p_t\varphi}\bigg( (-ik)^2 \bigg\{\p_t\varphi q(z,\nabla_z\varphi)a'_l \bigg\} + (-ik) \bigg\{ \p_t\varphi La'_l + A_l^{'(0)} \bigg\}  + \bigg\{A_l^{'(1)} - B_l^{'(0)} - \p_t\varphi Qa'_l\bigg\} 
     \\ & \qquad\qquad
     + \sum_{j=1}^{N-1} (-ik)^{-j} \bigg\{ A_l^{'(j+1)} - B_l^{'(j)} + C_l^{'(j-1)} \bigg\} \bigg) 
    \\ & \quad
    +\frac{e^{ik\varphi(0)}}{\p_t\varphi(0)}\bigg( (-ik) \bigg\{ A_l^{'(0)}(0)   \bigg\} + \bigg\{  \big(A_l^{'(1)} - B_l^{'(0)}\big)(0)   \bigg\} + \sum_{j=1}^{N-1} (-ik)^{-j} \bigg\{ \big(A_l^{'(j+1)} - B_l^{'(j)} + C_l^{'(j-1)}\big)(0) \bigg\} \bigg) 
    \\ & \quad + O(k^{-N}).
\end{align*}
Summing over the index $l$, one obtains
\begin{align*}
    & \sum_{l=0}^{N-1} (-ik)^{-l} P(e^{ik\varphi(x,t)}a'_l(x,t)) =
                 \\ & =
      -\frac{e^{ik\varphi}}{\p_t\varphi}\bigg((-ik)^{2} \bigg\{\p_t\varphi q(z,\nabla_z\varphi)a'_0 \bigg\} +  (-ik) \bigg\{\p_t\varphi (q(z,\nabla_z\varphi)a'_1 + La'_0) + A_0^{'(0)}\bigg\} 
      \\ & \qquad
      +   \bigg\{\p_t\varphi (q(z,\nabla_z\varphi)a'_2 +  La'_1 - Qa'_0) + \sum_{j=0}^{1} A_{j}^{'(1-j)} - B_0^{'(0)} \bigg\} 
      \\ & \qquad
      + \sum_{l=1}^{N-1} (-ik)^{-l} \bigg\{\p_t\varphi (q(z,\nabla_z\varphi)a'_{l+2} + La'_{l+1}- Qa'_l) + \sum_{j=0}^{l+1} A_{j}^{'(l+1-j)} - \sum_{j=0}^{l} B_{j}^{'(l-j)} +\sum_{j=0}^{l-1} C_j^{'(l-1-j)} \bigg\}  \bigg) 
    \\ & \quad
    +\frac{e^{ik\varphi(0)}}{\p_t\varphi(0)}\bigg((-ik) \bigg\{ A_0^{'(0)}(0)\bigg\}+   \bigg\{ \bigg(\sum_{j=0}^{1} A_{j}^{'(1-j)} - B_0^{'(0)}\bigg)(0) \bigg\} 
      \\ & \qquad
      + \sum_{l=1}^{N-1} (-ik)^{-l} \bigg\{ \bigg(\sum_{j=0}^{l+1} A_{j}^{'(l+1-j)} - \sum_{j=0}^{l} B_{j}^{'(l-j)} +\sum_{j=0}^{l-1} C_j^{'(l-1-j)} \bigg)(0)\bigg\} \bigg) 
    \\ & \quad + O(k^{-N}).
\end{align*}
Similarly, one gets
\begin{align*}
    & P(e^{ik\varphi(x,0)}a''_l(x,t))= 
    \\& =
     e^{ik\varphi(0)} \bigg( (-ik)^2 \bigg\{\int_0^t  A''_l(t,s) \,ds - c|\nabla\varphi(0)|^2 a''_l \bigg\} 
     \\ & \qquad \qquad
     - (-ik) \bigg\{\int_0^t  B''_l(t,s) \,ds - 2c \nabla \varphi(0) \cdot \nabla a''_l - \nabla \cdot c(x) \nabla \varphi(0)a''_l \bigg\} + \bigg\{\int_0^t  C''_l(t,s) \,ds+ Qa''_l\bigg\} \bigg),
\end{align*}
and summing over the index $l$ eventually gives
\begin{align*}
    &\sum_{l=0}^{N-1} (-ik)^{-l-1} P\big(e^{ik\varphi(x,0,\sigma)}a''_l(x,t,\sigma)\big) =
    \\ & =
    e^{ik\varphi(0)}\bigg(  (-ik) \bigg\{\int_0^t  A''_0(t,s) \,ds - c|\nabla\varphi(0)|^2 a''_0 \bigg\} 
     \\ & \qquad\quad
     + \bigg\{\int_0^t  \big( A''_1 - B''_0\big)(t,s) \,ds - c|\nabla\varphi(0)|^2 a''_1 + 2c \nabla \varphi(0) \cdot \nabla a''_0 + \nabla \cdot c(x) \nabla \varphi(0)a''_0 \bigg\} 
     \\ & \qquad \quad
     +\sum_{l=1}^{N-1} (-ik)^{-l} \bigg\{\int_0^t  \big(A''_{l+1}- B''_l + C''_{l-1}\big)(t,s) \,ds 
     \\ & \qquad\qquad\qquad\qquad\quad - c|\nabla\varphi(0)|^2 a''_{l+1}  + 2c \nabla \varphi(0) \cdot \nabla a''_l + \nabla \cdot c(x) \nabla \varphi(0)a''_l + Qa''_{l-1}\bigg\} .
\end{align*}

Therefore, coming back to $Pu$,
\begin{align*}
    & Pu(x,t) = 
                                   \\ & =
    \int_0^T
    -e^{ik\varphi}\bigg((-ik)^{2} \bigg\{ q(z,\nabla_z\varphi)a'_0 \bigg\} +  (-ik) \bigg\{ q(z,\nabla_z\varphi)a'_1 + La'_0 + \frac{A_0^{'(0)}}{\p_t\varphi}\bigg\} 
      \\ & \qquad\qquad\qquad
      +   \bigg\{q(z,\nabla_z\varphi)a'_2 +  La'_1 - Qa'_0 + \frac{\sum_{j=0}^{1} A_{j}^{'(1-j)} - B_0^{'(0)}}{\p_t\varphi} \bigg\} 
      \\ & \qquad\qquad\qquad
      + \sum_{l=1}^{N-1} (-ik)^{-l} \bigg\{q(z,\nabla_z\varphi)a'_{l+2} + La'_{l+1}- Qa'_l 
      \\ & \qquad\qquad\qquad\qquad\qquad\qquad\;
      + \frac{\sum_{j=0}^{l+1} A_{j}^{'(l+1-j)} - \sum_{j=0}^{l} B_{j}^{'(l-j)} +\sum_{j=0}^{l-1} C_j^{'(l-1-j)}}{\p_t\varphi} \bigg\}  \bigg) d\sigma
    \\ & \quad
    + \int_0^Te^{ik\varphi(0)}\bigg(  (-ik) \bigg\{V a''_0 + \frac{A_0^{'(0)}}{\p_t\varphi}(0)\bigg\} 
     +\bigg\{Va_1'' - \int_0^t  B''_0(t,s) \,ds + \tilde L a''_0 + \bigg(\frac{\sum_{j=0}^{1} A_{j}^{'(1-j)} - B_0^{'(0)}}{\p_t\varphi}\bigg)(0)\bigg\} 
     \\ & \qquad\qquad\qquad \quad
     +\sum_{l=1}^{N-1} (-ik)^{-l} \bigg\{V a''_{l+1} + \int_0^t  \big(- B''_l + C''_{l-1}\big)(t,s) \,ds   + \tilde L a''_l + Qa''_{l-1}
     \\ & \qquad\qquad\qquad\qquad\qquad\qquad \quad
     +\bigg(\frac{\sum_{j=0}^{l+1} A_{j}^{'(l+1-j)} - \sum_{j=0}^{l} B_{j}^{'(l-j)} +\sum_{j=0}^{l-1} C_j^{'(l-1-j)}}{\p_t\varphi} \bigg)(0)\bigg\} \bigg) d\sigma
    \\ & \quad + O(k^{-N})
    \\ & =: I_1+I_2+O(k^{-N}),
\end{align*}
where we set $\tilde L:= -c \nabla \varphi(0) \cdot \nabla -\frac 12 \nabla \cdot c(x) \nabla \varphi(0)$ and
$$V w(x,t) :=  \frac 12 c|\nabla\varphi(0)|^2 w(x,t) -\frac 12 \int_0^t  G(x,t-s) |\nabla \varphi(x,s)|^2 w(x,s) \,ds.$$

\textbf{Step 4 (Transport and Volterra equations).} By formula \eqref{infinite-order}, we know that the remainder $r(z,\sigma):= q(z,\nabla_z\varphi) +\p_\sigma\varphi$ vanishes to infinite order at $z(\sigma)$. Integrating by parts in $\sigma$ we get
\begin{align*}
    -\int_0^T e^{ik\varphi}q(z,\nabla_z\varphi)a'_{l} d\sigma =
    (-ik)^{-1}\int_0^T e^{ik\varphi} \p_\sigma a'_{l} d\sigma -\int_0^T e^{ik\varphi}ra'_{l} d\sigma -(-ik)^{-1} [e^{ik\varphi} a'_{l}]_0^T,
\end{align*}
so that coming back to the first integral $I_1$ we have
\begin{align*}
    I_1 & = -
       (-ik)^{2} \int_0^T e^{ik\varphi}ra' d\sigma -
       (-ik) [e^{ik\varphi} a']_0^T 
      \\ & \quad -
      \int_0^T
    e^{ik\varphi}\bigg((-ik) \bigg\{ Ta'_0\bigg\}+   \bigg\{Ta'_1 - Qa'_0 + \frac{ A_{0}^{'(1)} - B_0^{'(0)}}{\p_t\varphi} \bigg\} 
      \\ & \qquad\qquad\qquad
      + \sum_{l=1}^{N-1} (-ik)^{-l} \bigg\{ Ta'_{l+1}- Qa'_l + \frac{\sum_{j=0}^{l} (A_{j}^{'(1)} -  B_{j}^{'(0)})^{(l-j)} +\sum_{j=0}^{l-1} C_j^{'(l-1-j)}}{\p_t\varphi} \bigg\}  \bigg) d\sigma,
\end{align*}
where by definition $A^{'(0)}_l(t)=-\frac 12G(0)|\nabla\varphi|^2a'_l$ and $T:=L -\frac 12\left( \frac{G(0)|\nabla\varphi|^2}{\p_t\varphi}-\p_\sigma\right)$. The third term on the right-hand side of $I_1$ gives us a sequence of transport equations in the unknowns $a'_l$, each with a source term depending only on $\{a'_j\}_{j<l}$. Namely, we require that
\begin{equation}\label{transport-equations}
    \begin{cases}
    Ta'_0 &= 0, \\
    Ta'_1 & = Qa'_0 - \frac{ A_{0}^{'(1)} - B_0^{'(0)}}{\p_t\varphi}, \\
    Ta'_{l+1} & =Qa'_l -  \frac{\sum_{j=0}^{l} (A_{j}^{'(1)} -  B_{j}^{'(0)})^{(l-j)} +\sum_{j=0}^{l-1} C_j^{'(l-1-j)}}{\p_t\varphi} ,
    \end{cases}
\end{equation}
where all equalities must be understood as holding to infinite order at $(z(\sigma),\sigma)$. We can ensure that each solution $a'_l$ is supported in a small neighbourhood of $(z(\sigma),\sigma)$ by multiplying with a suitable cutoff function $\eta$ which is equal to $1$ close to $(z(\sigma),\sigma)$. Observe that due to the local nature of the operator this does not change the fact that the new amplitudes $\eta a'_l$ solve the transport equations above to infinite order at $(z(\sigma),\sigma)$. This ensures that
$$ I_1 = (-ik) (e^{ik\varphi(z,0)} a'(z,0)-e^{ik\varphi(z,T)} a'(z,T)) + O(k^{-\infty}). $$

We now turn to the second integral $I_2$ in the formula for $Pu$. Here we will require that the coefficients to all powers of $k$ vanish, which translates in a sequence of Volterra equations in the unknowns $a''_l$, with source terms depending on the functions $\{a'_j\}_{j\leq l}$ and $\{a''_j\}_{j< l}$. Because we have already solved the transport equations for the functions $a'_l$, we can find the $a''_l$ by solving the corresponding Volterra equations recursively. The transport and Volterra equations are solved as in the proof of Proposition \ref{geom-optics}, with the caveat that now the equations are intended to hold to infinite order on $(z(\sigma),\sigma)$. This implies that
$$ Pu = (-ik) (e^{ik\varphi(z,0)} a'(z,0)-e^{ik\varphi(z,T)} a'(z,T)) + O(k^{-N}).$$

This proves the existence of quasimodes in the general case. 

\textbf{Step 6 (Final estimate).} If we let 
$$r_N:= Pu - (-ik) (e^{ik\varphi(z,0)} a'(z,0)-e^{ik\varphi(z,T)} a'(z,T)),$$
we have proved that $|r_N|= O(k^{-N})$. Therefore, the $L^2$ estimate $\norm{r_N}_{ L^2((0,T), L^2(\Omega'))} = O(k^{-N})$
follows immediately by the boundedness of $\Omega'\times(0,T)$.
\end{proof}

Next, we shall use the above results to prove that propagation of singularities holds for our operator. We begin with a short reminder of some semiclassical concepts, following \cite{Zworski}. Let $h$ be a small parameter (in our case, we will assume $h:= k^{-1}$), and assume that $X\subset\mathbb R^n$ is open. A family $\{u_h\}_{h\in(0,1]}\subset L^2_{loc}(X)$ is said to be \emph{$L^2$-tempered} if for any $\chi\in C^\infty_c(X)$ there exists $N\geq 0$ such that $\|\chi u_h\|_{L^2}=O(h^{-N})$ as $h\rightarrow 0$. Let $\mathcal F_h$ denote the \emph{semiclassical Fourier transform} given by $$ \mathcal F_hf(\xi):=(2\pi h)^{-n}\int_{\mathbb R^n}e^{-i h^{-1}x\cdot\xi}f(x)dx.$$ We say that $u_h$ is \emph{semiclassically smooth} at $(x_0,\xi_0)\in T^*X$ if in some local coordinates near $x_0$ (the definition is independent of the choice of local coordinates) there exist $\phi,\psi\in C^\infty_c$ such that $$ \phi=1 \mbox{ near } x_0,\quad \psi=1 \mbox{ near } \xi_0,\quad \psi\mathcal F_h(\phi u_h)=O_{L^2}(h^\infty).$$ The \emph{semiclassical wave front set} $\WF_{scl}(u_h)$ is then defined as the complement of the set of those points $(x_0,\xi_0)\in T^*X$ where $u_h$ is semiclassically smooth. We will make use of the following characterization of the semiclassical wavefront set:

\begin{Lemma}{\cite[Proposition 2.3]{OSSU}}
    Let $u$ be $L^2$-tempered in $X$. If $z_0\in X$, one has $(z_0,\zeta_0)\notin WF_{scl}(u)$ if and only if
$$ ( u, e^{ik\Psi(\cdot;z,\zeta)}b )_{L^2(Z)} = O(k^{-\infty}) \text{ as $k \to \infty$}$$
uniformly over $(z,\zeta)$ near $(z_0,\zeta_0)$, where $\Psi,b$ are smooth functions such that in local coordinates near $z_0$ it holds
$$ \Psi(y;z,\zeta) = \zeta\cdot(y-z) + \frac{i}{2}|y-z|^2, \qquad \mbox{supp}(b) \mbox{ close to } z_0, \qquad b=1 \mbox{ near } z_0. $$
\end{Lemma}

\begin{proof}[Proof of Theorem \ref{propagation}]

Let $u$ be $L^2$-tempered in $M':= (-S',T')\times\Omega'$ with $S'<0<T<T'$, and assume that $(z_0,\zeta_0) \notin WF_{scl}(u)$. Recall that in our notation $\gamma(\sigma) := (z(\sigma),\zeta(\sigma))$ is the bicharacteristic curve through $(z_0,\zeta_0)$, which ends at $(z_T,\zeta_T)$. Given $(\widetilde z_T,\widetilde \zeta_T)$ close to $(z_T,\zeta_T)$, let $\widetilde\gamma$ be the bicharacteristic through $(\widetilde z_T,\widetilde \zeta_T)$, which begins at $(\widetilde z_0,\widetilde\zeta_0)$, close to $(z_0,\zeta_0)$. Let $b$ be any smooth function supported close to $\widetilde z_T$ such that $b=1$ near $\widetilde z_T$. We want to show that
$$ ( u, e^{ik\Psi(\cdot;\widetilde z_T,\widetilde \zeta_T)}b ) = O(k^{-\infty}) $$
holds uniformly with respect to the choice of $\widetilde \gamma$. We make use of our quasimode construction. Let $\widetilde a'$ be the solution to the backwards transport equation for the first amplitude term, such that for its final value it holds $b= ike^{ik(\widetilde\varphi(\cdot,T)- \Psi(\cdot; \widetilde\gamma(T)) )}\widetilde a'(\cdot,T)$. We define
$$\widetilde b(\cdot,\sigma):= ike^{ik(\widetilde\varphi(\cdot,\sigma)- \Psi(\cdot; \widetilde\gamma(\sigma)) )}\widetilde a'(\cdot,\sigma),$$
so that $\widetilde b(\cdot, T)=b$. Using Theorem \ref{GO-construction-general} and a Borel summation argument as in \cite[Th. 4.15]{Zworski}, we can find a solution $\widetilde v$ to
\begin{align*}
    P^*\widetilde v &= O(k^{-\infty}) + ik \big( e^{ik\widetilde \varphi(\cdot,T)}\widetilde a'(\cdot,T) - e^{ik\widetilde \varphi(\cdot,0)}\widetilde a'(\cdot,0) \big)
    \\ & =
    O(k^{-\infty}) + e^{ik\Psi(\cdot,\widetilde\gamma(T))}\widetilde b(\cdot,T) - e^{ik\Psi(\cdot,\widetilde \gamma(0))}\widetilde b(\cdot,0)
    \\ & =
    O(k^{-\infty}) + e^{ik\Psi(\cdot,\widetilde\gamma(T))}b -e^{ik\Psi(\cdot,\widetilde \gamma(0))}\widetilde b(\cdot,0).
\end{align*} 
Hence
$$ ( u, e^{ik\Psi(\cdot;\widetilde z_T,\widetilde \zeta_T)}b ) = O(k^{-\infty})  + ( u, P^*\widetilde v ) + (u,e^{ik\Psi(\cdot,\widetilde z_0,\widetilde \zeta_0)}\widetilde b(\cdot,0)) = O(k^{-\infty}),$$
where we used the fact that $(u,P^*\widetilde v) = (Pu,\widetilde v) =0$ and the assumption that $(z_0,\zeta_0)$ is not in the wavefront set of $u$. In order to check that this equality holds uniformly over $(\widetilde z_T,\widetilde \zeta_T)$ close to $(z_T,\zeta_T)$, we have to make sure that the function $\widetilde b$ and the coefficients of the $O(k^{-\infty})$ terms can be uniformly bounded for little variations of the bicharacteristic $\widetilde\gamma$, upon which they depend. Observe that the phase function $\widetilde \varphi$ and the amplitude $\widetilde a'$ appearing in $\widetilde b$ respectively arise as solutions to the matrix Riccati and transport equations. As discussed in \cite[Section 1.5]{Tay11}, see also \cite[Section 6]{MSS}, one sees that for all $N\in\mathbb N$ there exist constants $C, C'>0$ depending only on $c, G, T$ and $N$ such that the estimates
$$ \|\widetilde \varphi\|_{C^N}\leq C, \qquad \|\widetilde a'\|_{C^N}\leq C' $$
holds uniformly over $\widetilde\gamma$. We also want to show that the $C^N$-norm of $\widetilde a''$ is uniformly bounded. In order to do so, we come back to the estimate of $\widetilde a''$ as solution to the Volterra integral equation (see Step 4 in the proof of Proposition \ref{geom-optics}). We have
$$ \|\widetilde a''\|_{C^N} \lesssim \|f\|_{C^N}(1+\|R\|_{C^N}), $$
where the resolvent kernel $R$ and the source term $f$ verify $$\|R\|_{C^N}\lesssim \|\widetilde \varphi\|_{C^N}, \qquad \|f\|_{C^N} \lesssim \|\widetilde \varphi\|_{C^N} + \|\widetilde a'\|_{C^N}, $$
with the implied constants depending only on $c, G, T$ and $N$. Thus the uniform bound for $\widetilde a''$ follows from those for $\widetilde \varphi$ and $\widetilde a'$.
 One immediately sees by inspection that the terms in $O(k^{-\infty})$, which all arise either from $\widetilde\varphi, \widetilde a', \widetilde a''$, or from the fixed functions in the operator $P$, can also all be uniformly bounded with respect to $\widetilde\gamma$. This proves that $(z_T,\zeta_T)\notin WF_{scl}(u)$, which concludes the proof. 
 \end{proof}

 \begin{Remark}\label{stationary-sing}
     Besides the propagating singularities described above, which are analogous to the ones observable in the case of the wave equation with no memory term, solutions to the viscoacoustic wave equation are also know to present \emph{stationary singularities} (see Figure~\ref{fig:singularities}). Classical references of this fact can be found e.g. in \cite{HR84} for the one-dimensional case, for which one is able to explicitly compute solutions with stationary singularities, and in \cite{Ki94} for the general case (see also the more recent work \cite{BO22} on the related fractional Zener wave equation). Stationary singularities do not propagate: they are generated at $t=0$ and remain in the same position for all times. Thus, despite the fact that $(z_0,\zeta_0)\notin$ WF$(u)$, it is indeed still possible that there will be a stationary singularity at $z_T$. However, this does not contradict our previous result, given that the direction $\zeta$ of such singularity will be purely horizontal (that is, $\zeta=(\xi,0)$) due to its stationary nature, and thus it can not coincide with $\zeta_T$, which is bound to be orthogonal to the bicharacteristic $\gamma$. Propagating and stationary singularities are thus easily distinguishable by their directions. In our study of the inverse problem for the viscoacoustic wave equation we only make use of propagating singularities.
 \end{Remark}

\begin{center}
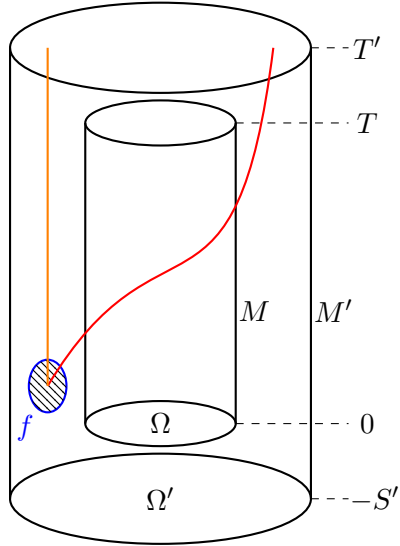
\begin{figure}[h]
     \begin{tikzpicture}

  \draw[thick] (0,1) ellipse (2 and 0.6); 
  \draw[thick] (-2,1) -- (-2,-5);          
  \draw[thick] (2,1) -- (2,-5);            
  \draw[thick] (0,-5) ellipse (2 and 0.6); 

  \draw[thick] (0,0) ellipse (1 and 0.3); 
  \draw[thick] (-1,0) -- (-1,-4);                      
  \draw[thick] (1,0) -- (1,-4);                        
  \draw[thick] (0,-4) ellipse (1 and 0.3);             

  \draw[dashed] (2,-5) -- (2.5,-5);
  \draw[dashed] (2,1) -- (2.5,1);
  \draw[dashed] (1,0) -- (2.5,0);
  \draw[dashed] (1,-4) -- (2.5,-4);

  \draw[blue,thick] (-1.5,-3.5) ellipse (0.25 and 0.35); 
  \fill[blue, pattern=north west lines] (-1.5,-3.5) ellipse (0.25 and 0.35);
  \draw[red,thick]
    (-1.5,-3.5) .. controls (0,-1) and (1,-3) .. (1.5,1);
  \draw[orange,thick]
    (-1.5,-3.5) -- (-1.5,1);

  \node at (0,-5) {$\Omega'$};
  \node at (0,-4) {$\Omega$};
  \node at (2.3,-2.5) {$M'$};
  \node at (1.25,-2.5) {$M$};
  \node at (-1.8,-4.05) {{\color{blue}$f$}};
  \node at (2.85,-5) {$-S'$};
  \node at (2.75,1) {$T'$};
  \node at (2.75,0) {$T$};
  \node at (2.75,-4) {$0$};

\end{tikzpicture}
\caption{Represented in red is a propagating singularity. An example of a stationary singularity is given in orange.}
\label{fig:singularities}

\end{figure}
\end{center}

 \section{Well-posedness of the direct problem and source-to-solution map} \label{sec_wellposedness}
 
 We study the well-posedness of the following problem:
 \begin{equation}\label{eq:initial-value-problem}
     \begin{cases}
         Pu = f, & \mbox{ in } \Omega\times[0,T]=: M, \\
         u(0) = u_0, \quad \dot u(0) = \dot u_0, & \mbox{ in } \Omega.
     \end{cases}
 \end{equation}
 For this we follow \cite[Appendix B]{dHLN} and \cite[Theorem 2.2]{Dafermos}. Observe that, as in the first reference, we shall define
 $$\mathbf C u(\cdot,t) := \nabla\cdot (c(\cdot)\nabla u(\cdot,t)), \qquad \mathbf G u(\cdot,t,s) := \nabla\cdot(G(\cdot,t-s)\nabla u(\cdot,s)), $$
 so that our equation $Pu=f$ can be rewritten in the form
$$ \p^2_t u(x,t) = \mathbf Cu(x,t) -\int_0^t \mathbf Gu(x,t,s)ds + f(x,t), $$
 as in the second reference. We make the following assumptions on the coefficients $\mathbf C, \mathbf G$ (see beginning of Section 2 in \cite{Dafermos}):
 \begin{itemize}
     \item $\mathbf C(t), \mathbf C_t(t) \in L^\infty([0,T], \mathcal L(H^1_0(\Omega), H^{-1}(\Omega)))$,
     \item $\langle \mathbf C(t) v, w\rangle = \langle \mathbf C(t) w, v\rangle, \quad \mbox{ for all } v,w\in H^1_0(\Omega),$
     \item there exists a constant $K>0$ such that
     $$ -\langle \mathbf C(t) v, v\rangle \geq K\|v\|_{H^1(\Omega)}^2, \quad \mbox{ for all } v\in H^1_0(\Omega),  $$
     \item $\mathbf G(t,s), \mathbf G_t(t,s) \in L^\infty([0,T]\times[0,T], \mathcal L(H^1_0(\Omega), H^{-1}(\Omega)))$.
 \end{itemize}
 Under these assumptions, we have
 \begin{Proposition}[Well-posedness, Theorems 2.1 and 2.2 in \cite{Dafermos}]\label{prop:well-posedness}
     Let $f\in L^2([0,T],{ H^{-1}(\Omega)})$, and $u_0=\dot u_0=0$. There exists a unique generalized solution $u\in L^2([0,T], H^1_0(\Omega)) \cap H^1([0,T], L^2(\Omega))$ to \eqref{eq:initial-value-problem}. Moreover, the following estimate holds:
     $$ \|u\|_{L^2([0,T], H^1_0(\Omega))} \lesssim \| f \|_{L^2([0,T],{ H^{-1}(\Omega)})}. $$
 \end{Proposition}
The proof of existence and uniqueness is obtained by applying the Lax-Milgram theorem on small time intervals. Given $t_0\in(0,T)$, one defines the bilinear form
$$ B_{t_0}(u,v) := \int_{0}^{t_0}(t-t_0)\left( \langle \dot u, \ddot v \rangle + \langle \nabla\cdot (c \nabla u ), v \rangle - \int_0^t \langle \nabla\cdot(G(t-s)\nabla u(s)), \dot v \rangle ds  \right)dt + \int_{0}^{t_0}\langle \dot u,\dot v\rangle dt, $$
where for the sake of simplicity we suppressed all dependencies on the variables $x,t$. Using the assumptions for $\mathbf C, \mathbf G$, one sees that $B_{t_0}$ is both coercive for small values of $t_0$, and bounded. Using the Riesz representation theorem, this guarantees the existence of a unique generalized solution in the time interval $[0,t_0]$, with the well-posedness estimate  
 $$ \|u\|_{L^2([0,t_0], H^1_0(\Omega))} \lesssim \| f \|_{L^2([0,t_0],{ H^{-1}(\Omega)})}. $$
The final result and estimate are obtained by repeating this argument a finite amount of times. 

 Consider now two spacetime cylinders $M\Subset M'$ with $M:=[-S,T]\times\Omega$, $M':=[-S',T']\times\Omega'$. The source-to-solution map associated to the operator $P$ and cylinders $M,M'$ is the map
 $$\mathcal S: C^\infty_c(M'\setminus M)\times H^1_0(\Omega')\times L^2(\Omega') \rightarrow L^2([-S',T'], H^1(\Omega'\setminus\Omega)) $$
 such that
 $$\mathcal S(f,u_0,\dot u_0) := u|_{M'\setminus M},$$
 where $u\in L^2([-S',T'], H^1_0(\Omega')) \cap H^1([-S',T'], L^2(\Omega'))$ is the unique solution to the problem
 \begin{equation}\label{eq:initial-value-problem-M'}
     \begin{cases}
         Pu = f, & \mbox{ in }  M', \\
         u(-S') = u_0, \quad \dot u(-S') = \dot u_0, & \mbox{ in } \Omega',
     \end{cases}
 \end{equation}
 whose existence is granted by Proposition \ref{prop:well-posedness}.

\section{Inverse problem} \label{sec_inverse}

 Let $\alpha_j$, $j=1,2$ be the scattering relation corresponding to sound speed $c_j$, which is defined for all $(z,\zeta)\in \p^\pm_j(T^*M):= \{(z,\zeta)=(x,t,\xi,\tau) \;|\; q_j(z,\zeta) := -(\tau^2 - c_j(x)|\xi|^2)=0, \tau_{j,\pm} (z,\zeta)=0\}$ as 
 $$ \alpha_j(z,\zeta) := \gamma_{z,\zeta}(\pm\tau_{j,\pm} (z,\zeta)), $$
 where $M:=[0,T]\times\Omega$ and $\tau_{j,\pm} (z,\zeta)$ is the travel time.
 We start by showing the following proposition:
 \begin{Proposition}\label{prop:from-cauchy-data-toscattering-relation}
     Let $M\Subset M'$ with $M:=[0,T]\times\Omega$, $M':=[-S',T']\times\Omega'$, and $\Omega,\Omega'\subset\mathbb R^n$ open, bounded and strictly convex. Let $c_j, G_j, P_j, \mathcal S_j$, $j=1,2$, respectively be speeds, viscous kernels, the operators defined by $$P_j u(x,t):= -\frac 12\left(\p^2_tu(x,t) -\nabla\cdot (c_j(x) \nabla u(x,t)) + \int_0^t \nabla\cdot(G_j(x,t-s)\nabla u(x,s))ds\right),$$ and the corresponding source-to-solution maps on the spacetime cylinders $M,M'$.
     Assume that $c_1=c_2$ to infinite order on $\p\Omega$, and that $G_1=G_2$ to infinite order on $\p M$. Moreover, assume that $\mathcal S_1 = \mathcal S_2$ holds. Then the scattering relations coincide, that is $\alpha_1=\alpha_2$. As a consequence, the travel times also coincide, i.e. $\tau_{1,\pm} = \tau_{2,\pm}$.
 \end{Proposition}

 For the proof of Proposition \ref{prop:from-cauchy-data-toscattering-relation} we follow \cite[Theorem 1.1]{OSSU}.

 \begin{proof}
  Because $P_1=P_2$ to infinite order on $\p M$, it is possible to smoothly extend the operators $P_1,P_2$ to a slightly larger cylinder $M'' \supset M$, $M'':= \Omega''\times [-S'', T'']$, in such way that $P_1=P_2$ in $M''\setminus M$. For the principal symbols $p_1,p_2$ we have $q_1 = q_2$ in $T^*(M''\setminus M)$.

  Observe that we have $\p^\pm_1(T^*M)=\p^\pm_2(T^*M)$. This is due to the fact that these sets only depend on $c_j$ on $\partial \Omega$, and by assumption we have that $c_1=c_2$ on $\p \Omega$ to infinite order. Thus from now on we will refer to both sets as just $\p^\pm(T^*M)$.

  Let $(z_0,\zeta_0)\in \p^\pm(T^*M)$, and let $\gamma_1:[0,T_1]\rightarrow T^*M\setminus 0$ be the maximal null bicharacteristic of $P_1$ in $M$ such that $\gamma_1(0)=(z_0,\zeta_0)$. Let $\widetilde\gamma_1:[-\widetilde S_1, \widetilde T_1]\rightarrow T^*M''$ be the maximal continuation of $\gamma_1$ to a null bicharacteristic of $P_1$ in $M''$, and let $\gamma_2:[-S_2,T_2]\rightarrow T^*M\setminus 0$ be the maximal null bicharacteristic of $P_2$ in $M$ such that $\gamma_2(0)=(z_0,\zeta_0)$. We will show that either $S_2=0$ or $T_2=0$, and also that
  $$\{\gamma_1(0), \gamma_1(T_1)\} = \{\gamma_2(-S_2), \gamma_2(T_2)\}.$$  This will mean that $\alpha_1 = \alpha_2$ on $\p^\pm(T^*M)$. 

  By Theorem \ref{GO-construction-general}, there exists a quasimode $v_1$ associated to $\widetilde \gamma_1$ in $M''$, such that
  $$\mbox{WF}_{scl}(v_1) = \widetilde \gamma_1([-\widetilde S_1, \widetilde T_1]), \qquad \mbox{WF}_{scl}(P_1v_1) = \widetilde \gamma_1(-\widetilde S_1)\cup \widetilde \gamma_1(\widetilde T_1).$$ 
  In particular $\|P_1v_1\|_{L^2([-S', T'], H^{-1}(\Omega'))} = O(k^{-\infty})$ for all cylinders $M':= \Omega'\times[-S', T']$ such that $M\subset M'\subset M''$, because $\mbox{WF}_{scl}(P_1v_1)$ is away from $\overline{M'}$. Let $r_1{ \in L^2([0,T], H^1_0(\Omega)) \cap H^1([0,T], L^2(\Omega))}$ solve $$P_1r_1 = -P_1v_1$$ in $M'$ with $\|r_1\|_{{  L^2}([-S',T'], H^1_0(\Omega'))}\lesssim \|P_1v_1\|_{L^2([-S', T'], H^{-1}(\Omega'))}$. Observe that the existence of $r_1$ is granted by \cite{Dafermos}, see our Proposition \ref{prop:well-posedness}. Then $$u_1 := v_1|_{M'} + r_1$$ verifies $P_1u_1=0$ in $M'$ and $\mbox{WF}_{scl}(u_1) = \mbox{WF}_{scl}(v_1) = \widetilde \gamma_1([-\widetilde S_1, \widetilde T_1])\cap T^*M'$.

{ 
  Let $u_2\in L^2([0,T], H^1_0(\Omega)) \cap H^1([0,T], L^2(\Omega))$ solve the problem 
  \begin{equation*}
     \begin{cases}
         P_2u_2 = 0, & \mbox{ in }  M', \\
         u_2(-S') = u_1(-S'), \quad \dot u_2(-S') = \dot u_1(-S'), & \mbox{ in } \Omega'.
     \end{cases}
 \end{equation*}
This problem has a unique solution in light of the well-posedness result of Proposition \ref{prop:well-posedness}, since $u_2$ can be written as $u_1 + v_2$, where $v_2$ solves the problem with vanishing Cauchy data and source $P_2(u_2-u_1)$.  
 }
 Because by assumption $\mathcal S_1 = \mathcal S_2$, we must have $u_2|_{M'\setminus M} = u_1|_{M'\setminus M}$. 
  Moreover, by the well-posedness estimate we see that
  \begin{align*}
      \|u_1- u_2\|_{L^2([-S',T'], H^1_0(\Omega'))} & \lesssim 
      \|P_2(u_1- u_2)\|_{L^2([-S',T'], H^{-1}(\Omega'))}
      \\ & =
      \|P_2u_1\|_{L^2([-S',T'], H^{-1}(\Omega'))}
      \\ & =
      \|(P_2-P_1)u_1\|_{L^2([-S',T'], H^{-1}(\Omega'))}
      \\ & \lesssim 
      \|u_1\|_{L^2([-S',T'], H^1_0(\Omega'))} ,
  \end{align*} 
  and thus
  $$\|u_2\|_{L^2([-S',T'], H^1_0(\Omega'))} \lesssim \|u_1\|_{L^2([-S',T'], H^1_0(\Omega'))},$$
  where the last term is $O(k^{2-\frac{n+1}4})$ by \cite[(3.38)]{OSSU}. Thus, we obtain that $u_2$ is $L^2$-tempered. Moreover, because $ u_2|_{M'\setminus M} = u_1|_{M'\setminus M}$, we have
  $$\mbox{WF}_{scl}(u_2|_{M'\setminus M}) = \bigg(\widetilde \gamma_1([-\widetilde S_1,0))\cup \widetilde \gamma_1((T_1,\widetilde T_1])\bigg)\cap T^*(M'\setminus M),$$
  and since the wavefront set is closed we deduce $\gamma_1(0), \gamma_1(T_1)\in \mbox{WF}_{scl}(u_2)$. However, we have $\gamma_1(0)=\gamma_2(0)$, and so the propagation of singularities result implies $\gamma_2(-S_2), \gamma_2(T_2)\in \mbox{WF}_{scl}(u_2)$. This allows us to deduce that
  $$\{\gamma_1(0), \gamma_1(T_1)\} = \{\gamma_2(-S_2), \gamma_2(T_2)\}.$$ In particular, $(z_0,\zeta_0)$ is one of the endpoints of $\gamma_2$, and therefore either $S_2=0$ or $T_2=0$.
  
  Having shown that the scattering relations coincide, we can easily deduce that the travel times $\tau_{1,\pm}, \tau_{2,\pm}$ also coincide, due to the particular expression that the bicharacteristics assume for our problem. To see this, let $(z,\zeta)\in\p^\pm(T^*M)$, and observe that $\gamma_{z,\zeta}(\sigma) = (x(\sigma), t+\sigma, \xi(\sigma), 1)$ for all $\sigma\in[0,T]$. Therefore, from 
  $$ \gamma_{z,\zeta}(\pm\tau_{1,\pm}(z,\zeta)) = \alpha_1(z,\zeta) = \alpha_2(z,\zeta) = \gamma_{z,\zeta}(\pm\tau_{2,\pm}(z,\zeta))$$ 
  we can deduce that $\tau_{1,\pm}(z,\zeta) = \tau_{2,\pm}(z,\zeta)$ by comparing just the time coordinates. 
 \end{proof}

Next we prove

\begin{Proposition}\label{formula}
    Let $\gamma:\sigma\mapsto (z(\sigma),\zeta(\sigma))$ be a fixed null bicharacteristic of $P$ in $M:=\Omega\times[0,T]$, and let $v$ be a quasimode associated to $\gamma$. Let $\sigma_0\in[0,T]$ and $z_0 =(x_0,t_0):= z(\sigma_0)$. Then
    \begin{equation}\label{eq:vzeta0}
        v(z_0) = k^{-1/2}\bigg( \frac{H(\sigma_0)\dot z(\sigma_0)\cdot \dot z(\sigma_0)}{2\pi i} \bigg)^{1/2}\exp\bigg[ -i\int_0^{\sigma_0} \frac{G(x(\sigma),0)|\nabla\varphi(z(\sigma),\sigma)|^2}{\p_t\varphi(z(\sigma),\sigma)}d\sigma \bigg] + O(k^{-3/2}).
    \end{equation}
\end{Proposition}
\begin{proof}
    We adapt Lemma 3.6 from \cite{OSSU}. By the quasimode construction, we have for all $z\in M$
    $$ v(z) = \int_0^T e^{ik\varphi(z,\sigma)}a'(z,\sigma)d\sigma + \int_0^T e^{ik\varphi(x,0,\sigma)}a''(z,\sigma)d\sigma, $$
    where $\varphi$ is as in equation \eqref{eq:define-varphi}, $a'(z,\sigma) = \sum_{l=0}^{N-1}(-ik)^{-l}a'_l(z,\sigma)$ is constituted by solutions of the transport equations \eqref{transport-equations}, and similarly $a''$ is constituted by solutions to Volterra integral equations. Given $\varepsilon>0$, define the following neighbourhood of $\sigma_0$ $$I_{\sigma_0}:= \{ \sigma\in[0,T] : |z(\sigma)-z_0|\leq \varepsilon \},$$
    and rewrite the formula above in the special case $z=z_0$ as
    $$ v(z_0) = \int_{I_{\sigma_0}} e^{ik\varphi(z_0,\sigma)}a'(z_0,\sigma)d\sigma + \int_{[0,T]\setminus I_{\sigma_0}} e^{ik\varphi(z_0,\sigma)}a'(z_0,\sigma)d\sigma + \int_0^T e^{ik\varphi(x_0,0,\sigma)}a''(z_0,\sigma)d\sigma.$$    
    Because in general we have
    $$ \mbox{Im}(\varphi(z,\sigma)) \geq c |z-z(\sigma)|^2$$
    for a fixed constant $c>0$, we can estimate
    $$ |e^{ik\varphi(z,\sigma)}| \leq |e^{ik\mbox{Im}(\varphi(z,\sigma))}| \lesssim e^{-ck|z-z(\sigma)|^2}. $$
    Thus we can conclude that the second integral is $O(e^{-\widetilde ck})$ for a constant $\widetilde c>0$. In a similar way, if we choose $\sigma_0$ sufficiently close to $T$ ($\sigma_0\geq T/2$ will suffice), for the third integral we can conclude
    $$ \int_0^T e^{ik\varphi(x_0,0,\sigma)}a''(z_0,\sigma)d\sigma \lesssim |e^{ik\mbox{Im}(\varphi(x_0,0,\sigma))}| \lesssim e^{-ck|(x_0,0)-(x(\sigma),\sigma)|^2} = e^{-ck( |x_0-x(\sigma)|^2 + \sigma^2)} \leq  e^{-\frac{ckT^2}4},$$
    that is, the third integral is also $O(e^{-\widetilde ck})$ for a constant $\widetilde c>0$. Therefore, we just need to estimate the first integral in the expression for $v(z_0)$. It is easy to check by the definition of $\varphi$ that
    $$ \varphi(z_0,\sigma_0)=0, \qquad \p_\sigma\varphi(z_0,\sigma_0)=0, \qquad \p^2_\sigma\varphi(z_0,\sigma_0) = \big( H(\sigma_0)\dot z(\sigma_0) - \dot\zeta(\sigma_0) \big)\cdot \dot z(\sigma_0), $$
    where for the second equality we have used that $\zeta(\sigma_0)\cdot\dot z(\sigma_0)=0$ by the Hamilton equation. Because we have $\dot z(\sigma_0)\neq 0$ and $\mbox{Im}(H(\sigma_0))>0$, we deduce that
    $$ \p^2_\sigma\varphi(z_0,\sigma_0)\neq 0, $$
    and thus we can use stationary phase to get 
    \begin{align*}
        v(z_0) & = k^{-1/2}\bigg( \frac{H(\sigma_0)\dot z(\sigma_0)\cdot \dot z(\sigma_0)}{2\pi i} \bigg)^{1/2}a'(z_0,\sigma_0) + O(k^{-3/2})
        \\ & =
        k^{-1/2}\bigg( \frac{H(\sigma_0)\dot z(\sigma_0)\cdot \dot z(\sigma_0)}{2\pi i} \bigg)^{1/2}a'_0(z_0,\sigma_0)+ O(k^{-3/2}).
    \end{align*}
    Here we used the observations $\dot z(\sigma_0)\neq 0$ and $z(\sigma)\neq z_0$ for all $\sigma\neq \sigma_0$, together with the continuity of $\sigma\mapsto z(\sigma)$ and the closedness of $[0,T]$, to deduce that $I_{\sigma_0}$ is a connected open neighbourhood of $\sigma_0$ for $\varepsilon$ small enough.   
    Observe that by the equation solved by $\varphi$, we know that $a'_0$ solves the transport equation 
    $$  (b(\sigma)-\p_\sigma)a'_0 := \bigg(\frac{G(x(\sigma),0)|\nabla\varphi(z(\sigma),\sigma)|^2}{\p_t\varphi(z(\sigma),\sigma)} -\p_\sigma\bigg)a'_0 = 0  $$
    at $z=z(\sigma)$. Thus we obtain 
    $$ a'_0(z_0,\sigma_0) = \exp\bigg[ -i\int_0^{\sigma_0} b(\sigma)d\sigma \bigg],$$
    which gives us the expression of $v(z_0)$ in the statement.  
    \end{proof}
With this result at hand, we are eventually ready to prove our main inverse problems result:

    \begin{proof}[Proof of Theorem \ref{main}]
    Because $\mathcal S_1 = \mathcal S_2$, by Proposition \ref{prop:from-cauchy-data-toscattering-relation} we deduce that the lens data for $c_1, G_1$ and $c_2, G_2$ coincide. We now apply either \cite{Muh77}, see also \cite{Cro91, PSU23}, in the case $n=2$ under geometric assumptions (A1) and (A2), or  \cite{SUV15} in the case $n\geq3$ under geometric assumptions (A1) and (A3). This reveals that the sound speeds $c_1$ and $c_2$ coincide, and as a consequence we also deduce that $\varphi_1=\varphi_2$. 
     Using the symbols from the previous Proposition \ref{formula}, assume that $\sigma_0\in[0,T]$ is so large that $z_0:=z(\sigma_0)\notin M$. Then $v_1(z_0)=v_2(z_0)$, and by taking the limit $k\rightarrow\infty$ in the formula \eqref{eq:vzeta0} we deduce that $a'_{0,1}(z_0,\sigma_0) = a'_{0,2}(z_0,\sigma_0)$, where
     $$a'_{0,j}(z_0,\sigma_0) := \exp\bigg[ -i\int_0^{\sigma_0} b_j(\sigma)d\sigma \bigg].$$ Hence there exists a function $\kappa: \gamma \mapsto \kappa(\gamma)\in\mathbb Z$ such that $\int_\gamma (b_1-b_2) = 2\pi\kappa(\gamma)$. Due to this integral identity, for a very short geodesic $\gamma$ we have $|\kappa(\gamma)|<1$, and thus $\kappa(\gamma)=0$ because $\kappa$ only takes values in the integers. On the other hand, if $\gamma_1$ is contained in a sufficiently small (uniformly with respect to the geodesics) neighbourhood of $\gamma_2$, we have $|\kappa(\gamma_1)-\kappa(\gamma_2)|<1$, and thus $\kappa(\gamma_1)=\kappa(\gamma_2)$. This implies that $\kappa\equiv 0$, and
    $$\int_0^{\sigma_0} \frac{(G_1-G_2)(x(\sigma),0)|\nabla\varphi(z(\sigma),\sigma)|^2}{\p_t\varphi(z(\sigma),\sigma)}d\sigma= \int_\gamma (b_1-b_2)=0.$$
    Solving the inverse problem for the geodesic x-ray transform now gives $b_1=b_2$, from which we deduce $G_1(x,0)=G_2(x,0)$ for all $x\in\Omega$.
    
    We now iterate this procedure. Assume that there exists $N\in\mathbb N$ such that we have already obtained $a'_{l,1}(z_0,\sigma_0) = a'_{l,2}(z_0,\sigma_0)$ for all $l< N$, and let
    \begin{align*}v_N(z) :&= (-ik)^{N}\bigg(v(z) - \int_{0}^T e^{ik\varphi(z,\sigma)}\sum_{l=0}^{N-1} (-ik)^{-l} a'_l(z,\sigma)d\sigma\bigg)
    \\ & =
    \int_0^T e^{ik\varphi(z,\sigma)}\sum_{l=0}^{N-1} (-ik)^{-l} a'_{N+l}(z,\sigma)d\sigma + (-ik)^{N}\int_0^T e^{ik\varphi(x,0,\sigma)}a''(z,\sigma)d\sigma.
    \end{align*}
    It is clear from the above expression that $v_{N,1}(z_0)=v_{N,2}(z_0)$ holds by the inductive assumption and the fact that $\mathcal S_1 = \mathcal S_2$. Using stationary phase as in the proof of the previous proposition, we observe that 
    \begin{align*}
        v_N(z_0) & =
        k^{-1/2}\bigg( \frac{H(\sigma_0)\dot z(\sigma_0)\cdot \dot z(\sigma_0)}{2\pi i} \bigg)^{1/2}a'_N(z_0,\sigma_0)+ O(k^{-3/2}),
    \end{align*}
    and thus by taking the limit $k\rightarrow\infty$ we deduce that $a'_{N,1}(z_0,\sigma_0) = a'_{N,2}(z_0,\sigma_0)$. By induction, it follows that $a'_{l,1}(z_0,\sigma_0) = a'_{l,2}(z_0,\sigma_0)$ holds for all $l\in\mathbb N$. Recall that $a'_{l,j}$ solves the transport equation
    $$  (b_j(\sigma)-\p_\sigma)a'_{l,j} = \mathcal R'_{l,j} , \qquad \mbox{ for all }\quad  l\in\mathbb N, \quad j=1,2, $$
    at $z=z(\sigma)$, and that we have already shown that $b_1=b_2$.  By equation \eqref{eq:face-of-the-amplitudes}, we deduce that $$ \mathcal R'_{l+1,1} = \mathcal R'_{l+1,2}, \qquad\mbox{for all }\quad l\in\mathbb N $$ by solving the inverse problem for the geodesic X-ray transform. In the case $l=0$, by equation \eqref{transport-equations} we have
\begin{align*}
    \mathcal R'_1 (x,t) & = Qa'_0(x,t) - \frac{\p_s  A'_{0}(x,t,t) - B'_{0}(x,t,t)}{\p_t\varphi},
\end{align*}
and recalling the definitions of $A', B'$ we recognize that all terms appearing in this equality are known, besides $\p_sG(x,0)$. Thus we can use the above equation to deduce that $\p_sG_1(x,0)=\p_sG_2(x,0)$ for all $x\in\Omega$. For all $l>0$, we see that the highest order $\p_s$ derivative involved in the computation of $\mathcal R'_{l+1}$ is $\p_s^{l+1}A'_0$. This means that for all $l>0$ the remainder $\mathcal R'_{l+1}$ is an expression in which the terms of the kind $\p_s^jG(x,0)$ appear up to order $l+1$, and all other terms are known. This allows us to use the amplitudes $a'_{l+1}$, $l\geq 0$, to recover the terms $\p_s^{l+1}G(x,0)$ by recursion, thus deducing that $\p_s^jG_1(x,0)=\p_s^jG_2(x,0)$ for all $j\in\mathbb N$ and all $x\in\Omega$. This completes the proof. 
\end{proof}

In light of the discussion in the Introduction, we finally observe that this amounts to recovering $\widetilde G(x,t)$ and all its time derivatives at time $t=0$.

\section{Applications}\label{sec:emm}

As an application of our results, we consider a scalar version of the extended Maxwell model (EMM) for viscoacoustic systems described in  \cite{dHKLN24, dHLN}. This is a microscopic model for viscoacoustic behaviour in polymeric materials, in which the natural energy collecting and dissipating properties of the medium are modeled by many elementary spring-dashpot units connected in parallel. Each spring-dashpot element represents the behaviour of a given chemical species, which collects (elasticity) and dissipates (viscosity) energy according to its specific properties. The integro-differential equation is
\begin{equation*}
        \p^2_t u(x,t) - \nabla\cdot \left( \widetilde G(x,0)\nabla u(x,t) + \int_0^t(\p_t\widetilde G)(x,t-s)\nabla u(x,s) ds\right)=0,
\end{equation*}
{ which corresponds to our viscoacoustic operator $P$. The function $\tilde G$ is named \emph{relaxation function}, and experimental evidence suggests that for each elementary spring-dashpot unit it takes the form
$$ \tilde G(x,t) = \tilde G(x,0)e^{-t/\tau(x)},$$
where $\tau(x):=\frac{\eta(x)}{\tilde G(x,0)}$ is the \emph{relaxation time}, $\tilde G(x,0)$ is the \emph{spring constant}, and $\eta(x)$ is the \emph{viscous constant}. The assumption that the relaxation function depends on time as an inverse exponential is aptly named \emph{fading memory}: the most recent deformations have a much heavier effect than the least recent ones. For the \emph{extended} Maxwell model one assumes
$$\widetilde G(x,t)= \sum_{j=1}^N e^{t\alpha_j(x)}\beta_j(x), \qquad \alpha_j(x):= -\tau_j^{-1}(x), \quad \beta_j(x):= \tilde G_j(x,0) $$}
for some $N\in\mathbb N$ and coefficients $\alpha,\beta: \mathbb R^n\rightarrow \mathbb R^N$. In order to avoid repetitions, we assume that the coefficients $\alpha_j$ are all distinct.

\begin{figure}[t]
    \centering
    \begin{tikzpicture}[thick]

\tikzset{
  spring/.style={decorate,decoration={zigzag,segment length=4,amplitude=3}},
  dashpot/.style={draw,rectangle,minimum width=0.6cm,minimum height=0.8cm,fill=gray!20}
}

\draw (0,0) -- (0,-4);

\draw (4.6,0) -- (4.6,-4);

\draw (0,0) -- (1.8,0);
\draw (2.8,0) -- (4.6,0);
\draw[spring] (1.8,0) -- (2.8,0);
\node at (2.3,0.5) {{\color{red}$\widetilde G_{0}(0)$}};

\foreach \i/\y in {1/-1, 2/-2} {
    \draw (0,\y) -- (1,\y);
    \draw (2,\y) -- (3,\y);
    \draw[spring] (1,\y) -- (2,\y);
    \node at (1.5,\y+0.5) {{\color{red}$\widetilde G_{\i}(0)$}};
    \draw (3,\y) node[dashpot,anchor=west] (D\i) {};
    \draw (2.8,\y+0.4) -- (3,\y+0.4);
    \draw (2.8,\y-0.4) -- (3,\y-0.4);
    \node at (3.3,\y) {{\color{blue}$\eta_{\i}$}};
    \draw (D\i.east) -- (4.6,\y);
}

    \draw (0,-4) -- (1,-4);
    \draw (2,-4) -- (3,-4);
    \draw[spring] (1,-4) -- (2,-4);
    \node at (1.5,-3.5) {{\color{red}$\widetilde G_{N}(0)$}};
    \draw (3,-4) node[dashpot,anchor=west] (D4) {};
    \draw (2.8,-3.6) -- (3,-3.6);
    \draw (2.8,-4.4) -- (3,-4.4);
    \node at (3.3,-4) {{\color{blue}$\eta_{N}$}};
    \draw (D4.east) -- (4.6,-4);

\node at (2.3,-2.8) {$\vdots$};
\draw (0,-2) -- (-0.22,-2);
\draw (-0.3,-2) circle (0.08);
\draw (4.6,-2) -- (4.82,-2);
\draw (4.9,-2) circle (0.08);

\end{tikzpicture}
\end{figure}

In this framework, we want to prove the uniqueness result of Theorem \ref{th-recovering-EMM}. We start with the proofs of the following Lemmas.

\begin{Lemma}\label{lemma1-EMM}
    Let $N\in\mathbb N$, $k\in\mathbb N_0$, and assume that $\{\alpha_j\}_{j=1}^N,\{\beta_j\}_{j=1}^N\subset \mathbb R$. Then it holds
    \begin{equation}\label{eq:magic-formula}
        \sum_{j=1}^N \alpha_j^{N+k}\beta_j = -\sum_{l=0}^{N-1}\frac 1{l!} \left(\sum_{j=1}^N \alpha_j^{l+k}\beta_j\right) \p_y^{l}\left( \prod_{\nu=1}^N (y-\alpha_\nu) \right)_{y=0}.
    \end{equation} 
\end{Lemma}
\begin{proof}
Let $p(y):= \prod_{\nu=1}^N (y-\alpha_\nu)$. We immediately see that $p(\alpha_j)=0$ for all $j\in\{1,...,N\}$, and thus by Taylor's formula
\begin{align*}
    \sum_{j=1}^N \alpha_j^{N+k}\beta_j = -\sum_{j=1}^N \alpha_j^k\beta_j\left(p(\alpha_j) - \alpha_j^N\right) = -\sum_{j=1}^N \alpha_j^k\beta_j\left( \sum_{l=0}^\infty \frac{\p^lp(0)}{l!}\alpha_j^l - \alpha_j^N\right).
\end{align*}
However, we also see that $\p^lp\equiv 0$ for all $l>N$, and $\p^Np\equiv N!$. Therefore
\begin{align*}
    \sum_{j=1}^N \alpha_j^{N+k}\beta_j & = -\sum_{j=1}^N \alpha_j^k\beta_j\left( \sum_{l=0}^{N-1}\frac{\p^lp(0)}{l!}\alpha_j^l \right) = -\sum_{l=0}^{N-1}\frac1{l!}\left(\sum_{j=1}^N \alpha_j^{l+k}\beta_j\right) \p^lp(0).
\end{align*}
\end{proof}

\begin{Lemma}\label{lemma2-EMM}
    Assume that the set $S\subseteq \mathbb R^n$ defined by $$S:= S_\beta\cup S_\alpha := \bigcup_{l=1}^N \beta_l^{-1}(0)\cup\bigcup_{j<k, \; j,k=1}^N (\alpha_j-\alpha_k)^{-1}(0)$$ has vanishing measure. Then the matrix $M$ defined by $M_{ij}:= \sum_{l=1}^N \alpha_l^{i+j-2}\beta_l$ is invertible almost everywhere.
\end{Lemma}

\begin{proof}
    It suffices to show that the determinant of $M$ is
    $$\det M = \prod_{l=1}^N\beta_l  \prod_{j<k, \; j,k=1}^N (\alpha_j-\alpha_k)^2.$$
    In order to see this, observe that $\det M$ is a polynomial in the variables $\alpha_1, ..., \alpha_N, \beta_1, ..., \beta_N$ such that each summand has exactly $N$ factors coming from $\beta$. If any of the $\beta$ factors is repeated, that is the summand is divisible by e.g. $\beta_l^2$, then in the determinant the same summand also appears with opposite sign. This is due to the existence of two permutations of opposite signs interchanging the two rows of $M$ (say, the $k$-th and the $j$-th) from which the $\beta_l$ factors come. Observe that this permutation leaves unchanged the associated power of $\alpha_l$, which is $j+k$ in both cases. Thus all summands with repeated $\beta$ factors vanish, and the determinant can be written as
    \begin{equation}\label{intermediate-decomposition-polynomial}
        \det M = \pi(\alpha)\prod_{l=1}^N\beta_l
    \end{equation}
    for some fixed polynomial $\pi$ in the variables $\alpha_1, ..., \alpha_N$. Observe that $\pi$ does not depend on $\beta$, and thus it must be a homogeneous and symmetric polynomial with respect to $\alpha$. Moreover, observe that the two sequences of coefficients $\alpha_1, \alpha_1, \alpha_3, ..., \alpha_N,\beta_1,\beta_2,\beta_3,...,\beta_N$ and $\alpha_1, \alpha_1, \alpha_3, ..., \alpha_N,0,(\beta_1+\beta_2),\beta_3,...,\beta_N$ give rise to the same matrix, and hence the same determinant, which clearly vanishes in the second case because of formula \eqref{intermediate-decomposition-polynomial}. Together with symmetry, this proves that $(\alpha_1-\alpha_2)$ divides $\pi(\alpha)$ an even number of times, and the same can be said for all polynomials of the kind $(\alpha_j-\alpha_k)$, with $j,k\in\{1,...,N\}$. Thus we must have
    $$\pi(\alpha)=\tilde\pi(\alpha)\prod_{j<k, \; j,k=1}^N (\alpha_j-\alpha_k)^2$$
    for some fixed polynomial $\tilde\pi$. In order to find the final polynomial $\tilde\pi$, we will compute the degrees of $\pi$ and of the product on the right-hand side of the above formula. The power with which $\alpha$ appears in the element $M_{ij}$ is $i+j-2$; because each summand of the homogeneous polynomial $\pi$ is produced by multiplying $\alpha$ factors coming from $N$ distinct rows and columns of $M$, we must have
    $$\deg \pi = \sum_{i,j=1}^N(i+j-2) = 2 \sum_{i=1}^N i - 2N = N(N+1)-2N = N(N-1).$$
    As for the product on the right-hand side, we see that it consists of $\frac{N(N-1)}2$ quadratic factors, and so its degree is also $N(N-1)$. This proves that $\tilde \pi$ is a constant. Next, we will show one example in which we are able to compute the determinant easily, which will establish that $\tilde\pi \equiv 1$ and complete the proof. 
    
    Let $V$ be the Vandermonde matrix given by $V_{ij}:= \alpha_j^{i-1}$, and consider the linear system $$ V\beta = e_N, $$ where $e_N$ is the $N$-th standard coordinate vector. Because $V$ is invertible in $\mathbb R^n\setminus S_\alpha$, by the properties of the Vandermonde matrix we can compute $\beta$ as the unique solution given by 
    $$ \beta = V^{-1}e_N = \left(\begin{matrix}
        L_{0,0} & ... & L_{N-1,0} \\ \vdots &  & \vdots \\ L_{0,N-1} & ... & L_{N-1,N-1} \end{matrix} \right)e_N = (L_{N-1,0},...,L_{N-1,N-1}), $$
    where for all $j\in\{0,...,N-1\}$ the function $L_j$ is the Lagrange polynomial $$L_j(y):= \prod_{i\neq j,\; i=0}^{N-1}\frac{y-\alpha_{i+1}}{\alpha_{j+1}-\alpha_{i+1}} =: \sum_{i=0}^{N-1} L_{i,j}y^i.$$
    In particular, we have obtained that for all $l\in\{1,...,N\}$ it holds
    $\beta_l = \prod_{i\neq l,\; i=1}^{N}\frac{1}{\alpha_l-\alpha_i}$. Since this $\beta$ solves the linear system, we see that the first row of the matrix $M$ is $e_N$, and thus because of its a priori structure we have that $M$ is lower anti-triangular with $1$ along the non-principal diagonal. Therefore, $\det M =  (-1)^{\frac{N(N-1)}2}$. Computing $\det M$ with formula $\det M = \tilde\pi\prod_{l=1}^N\beta_l  \prod_{j<k, \; j,k=1}^N (\alpha_j-\alpha_k)^2$, we obtain
    \begin{align*}
        \det M & = \tilde\pi \prod_{i\neq l,\; i,l=1}^{N}\frac{1}{\alpha_l-\alpha_i}  \prod_{j<k, \; j,k=1}^N (\alpha_j-\alpha_k)^2 =
        \tilde\pi \prod_{i< l,\; i,l=1}^{N}  \prod_{j<k, \; j,k=1}^N  \frac{\alpha_j-\alpha_k}{\alpha_l-\alpha_i} =
        (-1)^{\frac{N(N-1)}2} \tilde\pi
    \end{align*}
    Thus, $\tilde\pi\equiv 1$.
\end{proof}

{  We are now ready to prove Theorem \ref{th-recovering-EMM}.

\begin{proof}[Proof of Theorem \ref{th-recovering-EMM}]
Observe that for all $k\in\{0,...,2N-1\}$ and $x\in\mathbb R^n$ it holds
\begin{equation}\label{eq:assumed-data}
   \sum_{j=1}^N \alpha_{1,j}^k\beta_{1,j} = (\p_t^k \widetilde G_1)(\cdot,0) = (\p_t^k \widetilde G_2)(\cdot,0)= \sum_{j=1}^N \alpha_{2,j}^k\beta_{2,j}.
\end{equation} 
Using formula \eqref{eq:magic-formula}, we obtain that for all $k\in\{1,...,N\}$ it holds
$$\sum_{l=1}^{N}\frac 1{(l-1)!} \left(\sum_{j=1}^N \alpha_{1,j}^{l+k-2}\beta_{1,j}\right) \p^{l-1}p_1(0)  =\sum_{l=1}^{N}\frac 1{(l-1)!} \left(\sum_{j=1}^N \alpha_{2,j}^{l+k-2}\beta_{2,j}\right) \p^{l-1}p_2(0),$$
where $p_\mu(y):= \prod_{\nu=1}^N (y-\alpha_{\mu,\nu})$ for $\mu=1,2$. Because the sums in parentheses coincide by assumption, we can define a matrix $M$ as
$$ M_{lk} := \sum_{j=1}^N \alpha_{1,j}^{l+k-2}\beta_{1,j} = \sum_{j=1}^N \alpha_{2,j}^{l+k-2}\beta_{2,j}. $$
Observe that $M$ is invertible almost everywhere by Lemma \ref{lemma2-EMM}. Thus, if we define the vectors $c_\mu$ by $c_{\mu,l}:= \frac{\p^{l-1}p_\mu(0)}{(l-1)!}$ for $\mu=1,2$, we see that it must hold
$$ c_1 = M^{-1}Mc_2 = c_2. $$
Because $c_{\mu,l}$ is by definition the coefficient of the monic polynomial $p_\mu$ corresponding to the $(l-1)$-th power, we deduce that it must be $p_1 = p_2$. Thus $\alpha_{1,j}=\alpha_{2,j}$ for all $j\in\{1,...,N\}$, as these are the roots of the identical polynomials $p_1,p_2$. 

We are left with the problem of showing the equality of the $\beta$ coefficients. To this end, we define the Vandermonde matrix $V$ as
$$V_{ij}:= \alpha_{1,j}^{i-1} = \alpha_{2,j}^{i-1},$$
and observe that by rewriting the first $N$ equations from \eqref{eq:assumed-data} in matrix form we obtain
$$ V\beta_1 = V\beta_2. $$
Because $V$ is invertible for all $x\in\mathbb R^n\setminus S_\alpha$, we deduce that indeed $\beta_{1,j}=\beta_{2,j}$ for all $j\in\{1,...,N\}$. 
\end{proof}

\appendix
\section{Derivation of the viscoacoustic equation} \label{sec_appendix}

In this short appendix, for the convenience of the reader, we show the derivation of the viscoacoustic operator from the classical theory of viscoacousticity developed in \cite{Christensen}. Let $\sigma$ and $u$ respectively be a vector function and a scalar function of space and time, respectively indicating the time derivative of the particle velocity and pressure. We assume that there exists} a linear relation between $\sigma$ and $\nabla u$, and that $\sigma(t)$, $t\in[0,T]$ depends on all the history of $\nabla u$ up to time $t$. Then, by the Riesz representation theorem there exists a function $\widetilde G(x,t)$ such that $\sigma$ can be written as a Stieltjes integral:
$$\sigma(x,t) = \int_0^\infty \nabla u(x,t-s)\,d\widetilde G(x,s).$$
Here $\widetilde G(x,t)$ vanishes for all $t<0$. Assume that $\nabla u(x,t)$ also vanishes for all negative times. Then we can write
$$ \sigma(x,t) = \int_0^t \nabla u(x,t-s)\,d\widetilde G(x,s), $$
and using the relation between Stieltjes and Riemann integrals we have
\begin{align*}
    \sigma(x,t) & = \int_0^t \nabla u(x,t-s)\widetilde G'(x,s)ds 
    \\ & = 
    \int_0^t \nabla u(x,t-s)\bigg( \frac d{ds} \widetilde G(x,s) + \widetilde G(x,s)\delta(s) \bigg)ds
    \\ & = 
    \widetilde G(x,0)\nabla u(x,t)  +  
    \int_0^t \frac{d\widetilde G}{ds} (x,s) \nabla u(x,t-s) ds
    \\ & = 
    \widetilde G(x,0)\nabla u(x,t)  -  
    \int_0^t \frac{d\widetilde G}{ds} (x,t-s) \nabla u(x,s) ds.
\end{align*} 
Thus using $\partial^2_t u = \nabla\cdot\sigma$ we obtain the viscoacoustic equation
$$ \partial^2_t u - \nabla\cdot (\widetilde G(x,0)\nabla u(x,t))  +  
    \int_0^t \nabla\cdot\bigg(\frac{d\widetilde G}{ds} (x,t-s) \nabla u(x,s)\bigg) ds = \partial^2_t u - \nabla\cdot\sigma = 0. $$
In what follows we will let
$$c(x) := \widetilde G(x,0), \qquad G(x,t):= \frac{d\widetilde G}{dt} (x,t), $$
and we will write the equation as $Pu=0$, where $P$ is the (scalar) viscoacoustic operator given by
$$Pu := -\frac 12\left(\p_t^2 u - \nabla \cdot (c(x) \nabla u) + \int_0^t \nabla \cdot (G(x,t-s) \nabla u(x,s)) \,ds\right).$$

\bibliographystyle{alpha}
\bibliography{main-refs.bib}

\end{document}